\documentclass{amsart}
\usepackage{amsfonts}
\usepackage{amsmath}
\usepackage{amssymb}
\usepackage{graphicx}

\usepackage[bindingoffset=0.2in,%
            left=1.25in,right=1.25in,top=1.25in,bottom=1.375in,%
            footskip=.25in]{geometry}

\usepackage{amsthm}
\usepackage{latexsym}
\usepackage{fancyhdr}
\usepackage{amscd}
\usepackage{lscape}
\usepackage{tikz}
\usepackage{float}
\usepackage{mathtools}
\usepackage{lipsum}
\usepackage{bm}
\usepackage{subfigure}
\usepackage{booktabs}
\usepackage{grffile}
\usepackage{array}
\usepackage[tableposition=below]{caption}
\usepackage{longtable}
\newcolumntype{C}[1]{>{\centering\arraybackslash}p{#1}}
\definecolor{navy}{HTML}{2F729C} 
\usepackage[hyperfootnotes=false, colorlinks, linkcolor={blue}, citecolor={red}, filecolor={blue}, urlcolor={blue}, plainpages=false, pdfpagelabels]{hyperref}
\newtheorem{theorem}{Theorem}[section]
\newtheorem{lemma}[theorem]{Lemma}
\newtheorem{proposition}[theorem]{Proposition}
\newtheorem{corollary}[theorem]{Corollary}
\theoremstyle{definition}

\newtheorem{example}[theorem]{Example}

\newenvironment{claim}[1][Claim]{\begin{trivlist}
\item[\hskip \labelsep {\bfseries #1}]}{\end{trivlist}}

\theoremstyle{remark}

\numberwithin{equation}{section}

\begin{document}

\title[Good Elliptic Curves with a Specified Torsion Subgroup]{Good Elliptic Curves with a Specified Torsion Subgroup}

\author{Alexander J. Barrios}
\address{Department of Mathematics, University of St. Thomas, St. Paul, Minnesota 55105}
\email{abarrios@stthomas.edu}

\subjclass{Primary 11G05, 11D75, 11J25}



\begin{abstract}
An elliptic curve $E$ over $\mathbb{Q}$ is said to be good if $N_{E}^{6}<\max\!\left\{  \left\vert c_{4}^{3}\right\vert ,c_{6}^{2}\right\}  $ where $N_{E}$ is the conductor of $E$ and $c_{4}$ and $c_{6}$ are the invariants associated to a global minimal model of $E$. In this article, we generalize Masser's Theorem on the existence of infinitely many good elliptic curves with full $2$-torsion. Specifically, we prove via constructive methods that for each of the fifteen torsion subgroups $T$ allowed by Mazur's Torsion Theorem, there are infinitely many good elliptic curves $E$ with $E(\mathbb{Q})_{\text{tors}}\cong T$.
\end{abstract}
\maketitle

\section{Introduction}

By an $ABC$\textit{ triple}, we mean a triple $\left(  a,b,c\right)  $ such
that $a,b,$ and $c$ are relatively prime integers with $a+b=c$. Masser and
Oesterl\'{e}'s $ABC$ Conjecture \cite{MR992208} asserts that for each
$\epsilon>0$, there are finitely many $ABC$ triples satisfying
$\operatorname{rad}\!\left(  abc\right)  ^{1+\epsilon}<\max\!\left\{
\left\vert a\right\vert ,\left\vert b\right\vert ,\left\vert c\right\vert
\right\}  $. An $ABC$ triple is said to be \textit{good} if
$\operatorname{rad}\!\left(  abc\right)  <\max\!\left\{  \left\vert
a\right\vert ,\left\vert b\right\vert ,\left\vert c\right\vert \right\}  $. It
is well known that there are infinitely many good $ABC$ triples as demonstrated by
the $ABC$ triple $\left(  1,9^{n}-1,9^{n}\right)  $ which is good each
$n\geq1$ \cite{MR1860012}, \cite{MR3298238}.

In \cite{MR992208}, Oesterl\'{e} proved that the $ABC$ conjecture is
equivalent to the modified Szpiro conjecture, which states that for each
$\epsilon>0$, there are finitely many elliptic curves $E/%
\mathbb{Q}
$ such that $N_{E}^{6+\epsilon}<\max\!\left\{  \left\vert c_{4}^{3}\right\vert
,c_{6}^{2}\right\}  $ where $c_{4}$ and $c_{6}$ are the invariants associated
to a global minimal model of $E$ and $N_{E}$ denotes the conductor of $E$. As
with a good $ABC\ $triple, we define an elliptic curve $E/%
\mathbb{Q}
$ to be \textit{good} if $N_{E}^{6}<\max\!\left\{  \left\vert c_{4}%
^{3}\right\vert ,c_{6}^{2}\right\}  $. Masser \cite{MR1065152} proved that
there are infinitely many good elliptic curves having full $2$-torsion.
Bennett and Yazdani \cite{MR2931308} followed this by showing that there are
infinitely many good elliptic curves arising as rational points on twists of
the modular curve $X\!\left(  6\right)  $. We note that both of these results
concerned Szpiro's conjecture \cite{MR642675}. Specifically, they showed
the existence of infinitely many elliptic curves $E$ satisfying $N_{E}^{6}<\left\vert
\Delta\right\vert $ where $\Delta$ is the minimal discriminant of $E$. Since
$\left\vert \Delta\right\vert <\max\!\left\{  \left\vert c_{4}^{3}\right\vert
,c_{6}^{2}\right\}  $, we have that these elliptic curves are good. In
\cite{barrios2019constructive}, the author showed that there are infinitely
many good elliptic curves $E$ with $E\!\left(
\mathbb{Q}
\right)  _{\text{tors}}\cong C_{2}\times C_{2N}$ where $N=1,2,3,$ or $4$ and
$C_{m}$ denotes the cyclic group of order $m$. In this article, we extend
these results and prove:

\begin{claim}
[Theorem 1.]\textit{Let $T$ be one of the fifteen torsion subgroups allowed by Mazur's
Torsion Theorem~\cite{MR488287}. Then there are infinitely many good
non-isomorphic elliptic curves $E$ with $E(\mathbb{Q})  _{\text{tors}}\cong~T$.}
\end{claim}

Theorem $1$ is equivalent to Theorem \ref{goodtorsionmainthm}, where the
result is given in its constructive form. As a consequence of this result, we
have that for each nonnegative integer $n$, the elliptic curve%
\[
E_{n}:y^{2}+\left(  1-2^{n+15}\right)  xy-2^{n+15}%
=x^{3}-2^{n+15}x^{2}%
\]
is good with $E_{n}\!\left(
\mathbb{Q}
\right)  _{\text{tors}}\cong C_{5}$. We note that the method of proof
presented here for Theorem $1$ differs from the proof found in
\cite{barrios2019constructive} for $T=C_{2}\times C_{2N}$ with $N=1,2,3,$ or
$4$. This result relied on the full $2$-torsion structure of the models
considered, whereas our approach for Theorem $1$ will rely on the identity
$c_{4}^{3}-c_{6}^{2}=1728\Delta$ being a good $ABC$ triple. Specifically, we
consider parameterized families of elliptic curves $F_{T}=F_{T}\!\left(  a,b\right)  $ (see Table
\ref{ta:FTmodel}) where $T$ is one of the fifteen torsion subgroups $T$
allowed by Mazur's Torsion Theorem. In Section \ref{sec:models}, we show that
under certain assumptions on $a$ and $b$, it is the case that $F_{T}\!\left(
\mathbb{Q}
\right)  _{\text{tors}}\cong T$. In section \ref{sec:ABCtriples}, we study the
identity $1728\Delta_{T}+c_{6,T}^{2}=c_{4,T}^{3}$ where $c_{4,T}%
=c_{4,T}\!\left(  a,b\right)  $ and $c_{6,T}=c_{6,T}\!\left(  a,b\right)  $
are the invariants associated to a global minimal model of $F_{T}$ and
$\Delta_{T}=\Delta_{T}\!\left(  a,b\right)  $ is the minimal discriminant of
$F_{T}$. Under further assumptions on $a$ and $b$, we show that this is a good
$ABC$ triple. We conclude this section by showing that for each $T\neq
C_{1},C_{2},C_{5}$, there are relatively prime integers $a_{0}$ and $b_{0}$
such that the sequence $P_{n}^{T}=\left(  a_{n},b_{n},a_{n}+b_{n}\right)  $ of
$ABC$ triples defined recursively for $n\geq1$ by
\[
P_{n}^{T}=\left(  a_{n},b_{n},c_{n}\right)  =\left(  1728\Delta_{T}\!\left(
a_{n-1},b_{n-1}\right)  ,c_{6,T}\!\left(  a_{n-1},b_{n-1}\right)  ^{2}%
,c_{4,T}\!\left(  a_{n-1},b_{n-1}\right)  ^{3}\right)
\]
is a good $ABC$ triple. In Section \ref{InfManyGood}, we use these infinite
sequences of good $ABC$ triples to prove Theorem $1$. We also show the
following result, which allows us to find good $ABC$ triples from good
elliptic curves:

\begin{claim}
[Lemma 2.]Let $E$ be a good semistable elliptic curve with minimal
discriminant $\Delta\equiv0\ \operatorname{mod}6$. Let $c_{4}$ and $c_{6}$ be
the invariants associated to a global minimal model of $E$. Then, the $ABC$
triple $\left(  1728\Delta,c_{6}^{2},c_{4}^{3}\right)  $ is good.
\end{claim}

This is Lemma \ref{LemSS}, and we show that the assumption $\Delta
\equiv0\ \operatorname{mod}6$ is necessary. In Section~\ref{sec:examples}, we
construct good elliptic curves by considering all $ABC$ triples $\left(
a,b,c\right)  $ with $0<a<b<c\leq 1.2\cdot10^{11}$. This data is attained from the
ABC@Home Project \cite{Smit}, which computed all good $ABC$ triples with $c<10^{18}$.
For each $ABC$ triple with $c \leq 1.2\cdot10^{11}$, we create a database consisting
of the $%
\mathbb{Q}
$-isomorphism class of $F_{T}\!\left(  a,b\right)  $ and $F_{T}\!\left(
b,a\right)  $. By computing the modified Szpiro ratio of these elliptic
curves, we find $2 \ 347 \ 350$ good elliptic curves.

\section{Preliminaries}

We start by reviewing some facts and terminology about $ABC$ triples as well as elliptic
curves. For further details, see \cite{MR3584566}, \cite{MR2514094}.

By an $ABC$\textit{ triple}, we mean a triple $\left(  a,b,c\right)  $ of
integers such that $a+b=c$ and $\gcd\!\left(  a,b,c\right)  =1$. An
$ABC$ triple is said to be \textit{positive} if $a,b,c>0$. If $\operatorname{rad}%
\!\left(  abc\right)  <\max\!\left\{  \left\vert a\right\vert ,\left\vert
b\right\vert ,\left\vert c\right\vert \right\}  $, then the $ABC$ triple is
said to be \textit{good}. Here $\operatorname{rad}\!\left(  n\right)  $
denotes the product of the distinct prime factors of $n$. For instance, the
$ABC$ triple $\left(  1,8,9\right)  $ is a good positive $ABC$ triple since
$\operatorname{rad}\!\left(  1\cdot8\cdot9\right)  =6<9$. In~\cite{MR3298238},
it is shown that the $ABC$ triple $\left(  1,9^{n}-1,9^{n}\right)  $ is a good
$ABC$ triple for each positive integer $n$. In fact, this result generalizes
as follows:

\begin{lemma}
Let $p$ be an odd prime. Then $\left(  1,p^{\left(  p-1\right)  k}%
-1,p^{\left(  p-1\right)  k}\right)  $ is a good positive $ABC$ triple for
each positive integer $k$.
\end{lemma}

\begin{proof}
Note that $p^{\left(  p-1\right)  k}-1=\left(  p-1\right)  P$ where
$P=\sum_{j=1}^{\left(  p-1\right)  k}p^{j-1}$. Since $p\equiv\mp
1\ \operatorname{mod}\left(  p\pm1\right)  $, it follows that $P\equiv0$
$\operatorname{mod}\left(  p\pm1\right)  $. In particular, $P\equiv
0\ \operatorname{mod}8$ and thus%
\begin{equation}
\operatorname{rad}\!\left(  p^{\left(  p-1\right)  k}-1\right)
=\operatorname{rad}\!\left(  \frac{P}{4}\right)  \leq\frac{p^{\left(
p-1\right)  k}-1}{4\left(  p-1\right)  }.\label{exaabc1}%
\end{equation}
Moreover, $\operatorname{rad}\!\left(  \left(  p^{\left(  p-1\right)
k}-1\right)  p^{\left(  p-1\right)  k}\right)  =p\operatorname{rad}\!\left(
\left(  p^{\left(  p-1\right)  k}-1\right)  \right)  $. The result now follows
since%
\begin{align*}
p^{\left(  p-1\right)  k}-p\operatorname{rad}\!\left(  \left(  p^{\left(
p-1\right)  k}-1\right)  \right)   &  \geq p^{\left(  p-1\right)  k}%
-\frac{p^{\left(  p-1\right)  k+1}-p}{4\left(  p-1\right)  }\text{ by
(\ref{exaabc1})}\\
&  =p^{\left(  p-1\right)  k}\left(  1-\frac{p}{4\left(  p-1\right)  }\right)
+\frac{p}{4\left(  p-1\right)  }>0. \qedhere
\end{align*}
\end{proof}

For an $ABC$ triple $P=\left(  a,b,c\right)  $, we define the \textit{quality}
of $P$, denoted $q\!\left(  P\right)  $, to be the quantity%
\[
q\!\left(  P\right)  =\frac{\log\!\left(  \max\!\left\{  \left\vert
a\right\vert ,\left\vert b\right\vert ,\left\vert c\right\vert \right\}
\right)  }{\log\!\left(  \operatorname{rad}\!\left(  abc\right)  \right)  }.
\]
Consequently, an $ABC$ triple $P$ is good if and only if $q\!\left(  P\right)
>1$.

Next, we say an elliptic curve $E$ is defined over $%
\mathbb{Q}
$ if $E$ is given by an (affine) Weierstrass model%
\begin{equation}
E:y^{2}+a_{1}xy+a_{3}y=x^{3}+a_{2}x^{2}+a_{4}x+a_{6} \label{ch:inintroweier}%
\end{equation}
with each $a_{j}\in%
\mathbb{Q}
$. From the Weierstrass coefficients $a_{j}$, one defines the quantities%
\begin{equation}%
\begin{array}
[c]{l}%
c_{4}=a_{1}^{4}+8a_{1}^{2}a_{2}-24a_{1}a_{3}+16a_{2}^{2}-48a_{4},\\
c_{6}=-\left(  a_{1}^{2}+4a_{2}\right)  ^{3}+36\left(  a_{1}^{2}%
+4a_{2}\right)  \left(  2a_{4}+a_{1}a_{3}\right)  -216\left(  a_{3}^{2}%
+4a_{6}\right)  ,\\
\Delta=\frac{c_{4}^{3}-c_{6}^{2}}{1728}.
\end{array}
\label{basicformulas}%
\end{equation}
We call $\Delta$ the \textit{discriminant} of $E$, and the quantities $c_{4}$
and $c_{6}$ are the \textit{invariants associated to the Weierstrass model} of
$E$. Observe that from (\ref{basicformulas}), we have the identity
$1728\Delta=c_{4}^{3}-c_{6}^{2}$. By the Mordell-Weil Theorem, the group of
rational points on $E$, denoted by $E\!\left(
\mathbb{Q}
\right)  $, is a finitely generated abelian group. We denote the torsion
subgroup of $E\!\left(
\mathbb{Q}
\right)  $ by $E\!\left(
\mathbb{Q}
\right)  _{\text{tors}}$. By Mazur's Torsion\ Theorem \cite{MR488287}, there
are precisely fifteen possibilities for $E\!\left(
\mathbb{Q}
\right)  _{\text{tors}}$.

\begin{theorem}
[Mazur's Torsion Theorem]\label{torsiontheorem}Let $E$ be an elliptic curve defined
over $%
\mathbb{Q}
$. Then%
\[
E\!\left(
\mathbb{Q}
\right)  _{\text{tors}}\cong\left\{
\begin{array}
[c]{ll}%
C_{N} & \text{for }N=1,2,\ldots,10,12,\\
C_{2}\times C_{2N} & \text{for }N=1,2,3,4.
\end{array}
\right.
\]

\end{theorem}

We say $E$ is given by an \textit{integral Weierstrass model} if each
Weierstrass coefficient $a_{j}\in%
\mathbb{Z}
$. Two curves $E$ and $E^{\prime}$ are $\mathbb{Q}$\textit{-isomorphic} if $E^{\prime}$ can be attained from $E$ via an
admissible change of variables $x\longmapsto u^{2}x+r$ and $y\longmapsto
u^{3}y+u^{2}sx+w$ for $u,r,s,w\in\mathbb{Q}$ and $u\neq0$. Each elliptic curve $E/\mathbb{Q}$ is $\mathbb{Q}$-isomorphic to a \textit{global minimal model} $E^{\text{min}}$ where
$E^{\text{min}}$ is given by an integral Weierstrass model with the property
that its discriminant $\Delta$ satisfies the condition that $\left\vert
\Delta\right\vert $ is minimal over all $\mathbb{Q}$-isomorphic elliptic curves to $E$ that are given by an integral Weierstrass model. We call the discriminant of
$E^{\text{min}}$ the \textit{minimal discriminant} of $E$.

Next, let $\Delta$ and $c_{4}$ be associated to a global minimal model of $E$.
We say that $E$ has \textit{additive reduction} at a prime $p$ if $p$ divides
$\gcd\!\left(  \Delta,c_{4}\right)  $. If $E$ does not have additive reduction
at $p$, then $E$ is said to be \textit{semistable at }$p$. We say $E$ is
\textit{semistable} if $E$ is semistable at all primes. The \textit{conductor}
$N_{E}$ of $E$ is the quantity%
\[
N_{E}=\prod_{p|\Delta}p^{f_{p}}%
\]
where $f_{p}$ is a positive integer. The conductor exponent $f_{p}$ is
computed via Tate's Algorithm \cite{MR0393039} and has the property that
$f_{p}>1$ if and only if $E$ has additive reduction at $p$. In fact,
$f_{p}\leq2$ if $p\geq5$. As a consequence, we have that if $E$ is semistable,
then $N_{E}=\operatorname{rad}\!\left(  \Delta\right)  $. An elliptic curve
$E$ is \textit{good} if $N_{E}^{6}<\max\!\left\{  \left\vert c_{4}%
^{3}\right\vert ,c_{6}^{2}\right\}  $ where $c_{4}$ and $c_{6}$ are the
invariants associated to a global minimal model of $E$. The \textit{modified
Szpiro ratio} of $E$ is the quantity%
\[
\sigma_{m}\!\left(  E\right)  =\frac{\log\!\left(  \max\!\left\{  \left\vert
c_{4}^{3}\right\vert ,c_{6}^{2}\right\}  \right)  }{\log\!\left(
N_{E}\right)  }.
\]
In particular, the modified Szpiro ratio is the analog of the quality of an
$ABC$ triple, and we have that $E$ is good if and only if $\sigma_{m}\!\left(
E\right)  >6$.

\section{Models of Elliptic Curves\label{sec:models}}

For each torsion subgroup $T$ allowed by Theorem \ref{torsiontheorem}, let
$F_{T}=F_{T}\!\left(  a,b\right)  $ be the elliptic curve given by the
Weierstrass model found in Table \ref{ta:FTmodel}.
We note that the families of elliptic curves $F_T$ do not correspond to universal elliptic curves with torsion subgroup containing $T$. Instead, the family $F_T$ has the property that under certain conditions on the parameters, we have $F_T(\mathbb{Q})_\text{tors}\cong T$ (see Theorem \ref{torsiso}).
Next, let $A_{T}=A_{T}\!\left(
a,b\right)  ,\ B_{T}=B_{T}\!\left(  a,b\right)  ,$ and $D_{T}=D_{T}\!\left(
a,b\right)  $ be as given in Tables \ref{ta:AT}, \ref{ta:BT}, and \ref{ta:DT},
respectively. Our first result establishes that under certain assumptions on
$a$ and $b$, $D_{T}$ is the minimal discriminant of $F_{T}$ and the invariants
$c_{4}$ and $c_{6}$ associated to a global minimal model of $F_{T}$ are
$A_{T}$ and $B_{T}$, respectively. Consequently, we have the identity
$A_{T}^{3}-B_{T}^{2}=1728D_{T}$. We conclude this section by showing that
$F_{T}\!\left(
\mathbb{Q}
\right)  _{\text{tors}}\cong T$.

\begin{proposition}
\label{Propmin}For $T=C_{5}$, let $b=2^{n}$ for $n$ a nonnegative integer. For
$T\neq C_{5}$, let $a$ and $b$ be relatively prime integers with
$a\equiv0\ \operatorname{mod}6$. 

$\left(  1\right)  $ For each $T$, $F_{T}$ is given by an integral Weierstrass
model and $T\hookrightarrow F_{T}\!\left(
\mathbb{Q}
\right)  $;

$\left(  2\right)  $ The minimal discriminant of $F_{T}$ is $D_{T}$. The
invariants $c_{4}$ and $c_{6}$ associated to a global minimal model of $F_{T}$
are $A_{T}$ and $B_{T}$, respectively;

$\left(  3\right)  $ Suppose further that $a\equiv
0\ \operatorname{mod}5$ (resp. $a\equiv0\ \operatorname{mod}7$) if $T=C_{10}$
(resp. $T=C_{7}$). If $T\neq C_{5}$, then $F_{T}$ is semistable. For
$T=C_{5}$, $F_{T}$ is semistable at all primes $p\neq5$ and has additive
reduction at $5$. Moreover,
\[
\gcd\!\left(  A_{T}^{3},B_{T}^{2}\right)  =\left\{
\begin{array}
[c]{cl}%
1 & \text{if }T\neq C_{5},\\
125 & \text{if }T=C_{5}.
\end{array}
\right.
\]

\end{proposition}

\begin{proof}
By inspection, each $F_{T}$ is given by an integral Weierstrass model. Now suppose $T=C_{1}$. Then $T\hookrightarrow F_{T}\!\left(\mathbb{Q}\right)$ and from (\ref{basicformulas}), it is verified that the discriminant of $F_{T}$ is $D_{T}$ and that the invariants $c_{4}$ and $c_{6}$ of $F_{T}$ are $A_{T}$ and $B_{T}$, respectively. This verification can be found in \cite[Verification\_of\_Proposition\_3\_1.ipynb]{GitHubGoodEC}.
Next, consider $A_{T}\!\left(  a,b\right)  $ and $B_{T}\!\left(  a,b\right)  $ as polynomials in $R=\mathbb{Z}\!\left[  a,b,r,s\right]  $. Then there are $\mu_T,\nu_T\in R$ such that the following identity holds in $R$: $\mu_T A_{T}\!\left(  a,b\right)  +\nu_T B_{T}\!\left(  a,b\right)  =3^{33}\left(sa^{47}+rb^{47}\right)  $. The expressions for $\mu_T$ and $\nu_T$ are found in \cite[definitions.sage]{GitHubGoodEC}, and the identity is verified in \cite[Verification\_of\_Proposition\_3\_1.ipynb]{GitHubGoodEC}.  As a consequence of this identity, $\gcd\!\left(  A_{T},B_{T}\right)  $ divides $3^{33}$ since $a$ and $b$ are relatively prime. Since $a$ is divisible by $3$, it follows that $A_{T}\equiv b^{12}\ \operatorname{mod}3$ and thus $\gcd\!\left(  A_{T},B_{T}\right)  =1$. This shows that $F_{T}$ is given by a global minimal model which concludes the proof for $T=C_{1}$.

For $T\neq C_1$, let $\mathfrak{A}_{T}$ and$\ \mathfrak{B}_{T}$ be the integers defined in Table~\ref{frakAB} and let $\mathfrak{D}_{T}=1$ for each $T$ except $T=C_{2}$. For $T=C_{2}$, define $\mathfrak{D}_{C_{2}}=ab\left(  a+b\right)\left(  b-a\right)  $. Since $a$ and $b$ are relatively prime with $a$ even, it is easily verified that $\gcd\!\left(  \mathfrak{A}_{T},\mathfrak{B}_{T}\mathfrak{D}_{T}\right)  =1$ if $T\neq C_2,C_3,C_4,C_2\times C_2$. For the remaining $T$'s, we proceed as in the case for $T=C_1$ and note that there are $\mu_T,\nu_T\in R$ such that the following identities holds in $R$:
\begin{equation}
\mu_{T}\mathfrak{A}_{T}+\nu_{T}\mathfrak{B}_{T} \mathfrak{D}_T =\left\{
\begin{array}
[c]{cl}%
2^{15}\left(  sa^{17}+rb^{17}\right)   & \text{if }T=C_{2},\\
3^{6}\left(  sa^{17}+rb^{17}\right)   & \text{if }T=C_{3},\\
3^{4}\left(  sa^{15}+rb^{15}\right)   & \text{if }T=C_{4},\\
2^{12}\left(  sa^{15}+rb^{15}\right)   & \text{if }T=C_{2}\times C_{2},
\end{array}
\right.  \label{EucAlg}%
\end{equation}
Expressions for $\mu_T$ and $\nu_T$ are found in \cite[definitions.sage]{GitHubGoodEC}, and the identities in \eqref{EucAlg} are verified in \cite[Verification\_of\_Proposition\_3\_1.ipynb]{GitHubGoodEC}. Since $\gcd(a,b)=1$ with $a\equiv 0 \mod 6$, it follows that $\gcd\!\left(  \mathfrak{A}_{T},\mathfrak{B}_{T}\mathfrak{D}_{T}\right)  =1$ for each $T$.
{\renewcommand*{\arraystretch}{1.2}
\begin{longtable}{C{0.7in}C{2.4in}C{2.2in}}
\hline
$T$ & $\mathfrak{A}_{T}$ & $\mathfrak{B}_{T}$\\
\hline
\endfirsthead
\hline
$T$ & $\mathfrak{A}_{T}$ & $\mathfrak{B}_{T}$ \\
\hline
\endhead
\hline
\multicolumn{3}{r}{\emph{continued on next page}}
\endfoot
\hline
\caption{The quantities $\mathfrak{A}_{T}$ and $\mathfrak{B}_{T}$}\label{frakAB}
\endlastfoot
	
$C_{2}$ & $-(a^{8}-24a^{7}b+20a^{6}b^{2}-24a^{5}b^{3}-26a^{4}b^{4}%
+24a^{3}b^{5}+20a^{2}b^{6}+24ab^{7}+b^{8})$ & $8(a^{2}+b^{2})(a^{2}%
-2ab-b^{2})^{2}$\\\hline
$C_{3}$ & $(a^{3}-3a^{2}b-6ab^{2}-b^{3})^{3}$ & $-ab(a+b)(a^{2}+ab+b^{2})^{3}%
$\\\hline
$C_{4}$ & $(a-b)^{2}(a+b)^{6}$ & $-ab(a^{2}-ab+b^{2})^{3}$\\\hline
$C_{5}$ & $1$ & $32768b^{20}$\\\hline
$C_{6}$ & $(a^{2}-4ab+b^{2})^{2}$ & $-(a^{2}+b^{2})^{2}$\\\hline
$C_{7},C_{9},C_{10}$ & $a$ & $a+b$\\\hline
$C_{8}$ & $a^{2}+b^{2}$ & $b^{2}$\\\hline
$C_{12}$ & $a+b$ & $a$\\\hline
$C_{2}\times C_{2}$ & $-16ab(a-b)(a+b)(a^{2}+b^{2})^{2}$ & $(a^{2}%
-2ab-b^{2})^{4}$\\\hline
$C_{2}\times C_{4}$ & $(a^{2}-2ab-b^{2})^{2}$ & $ab(a+b)(a-b)$\\\hline
$C_{2}\times C_{6}$ & $a^{2}+b^{2}$ & $a^{2}+8ab+b^{2}$\\\hline
$C_{2}\times C_{8}$ & $a-b$ & $-\frac{1}{2}a$
\label{ta:ETcomp}	
\end{longtable}}
\vspace{-0.5em}

Next, let $E_{T}$ be the elliptic curve defined in \cite[Table 3]{2020arXiv200101016B} and let $\alpha_{T}  ,\ \beta_{T}  ,$ and $\gamma_{T} $ be as given in \cite[Tables~4,~5,~6]{2020arXiv200101016B}. By \cite[Lemma~2.9]{2020arXiv200101016B}, the invariants $c_{4}$ and $c_{6}$ of $E_{T}$ are $\alpha_{T}$ and $\beta_{T}$, respectively, and the discriminant of $E_{T}$ is $\gamma_{T}$. By \cite[Verification\_of\_Proposition\_3\_1.ipynb]{GitHubGoodEC}, we have that $F_{T}\!\left(  a,b\right)  =E_{T}\!\left(  \mathfrak{A}_{T},\mathfrak{B}_{T},\mathfrak{D}_{T}\right)  $ if $T=C_{2},C_{2}\times C_{2}$ and $F_{T}\!\left(  a,b\right)  =E_{T}\!\left(  \mathfrak{A}_{T},\mathfrak{B}_{T}\right)  $ if $T\not =C_{2},C_{2}\times C_{2}$. In particular, the discriminant of $F_T$ is $\gamma_T$ and the invariants $c_4$ and $c_6$ of $F_T$ are $\alpha_T$ and $\beta_T$, respectively. Specifically, we have that
\[
\left(  \alpha_{T},\beta_{T},\gamma_{T}\right)  =\left\{
\begin{array}
[c]{cl}%
\left(  \alpha_{T}\!\left(  \mathfrak{A}_{T},\mathfrak{B}_{T}\right)
,\beta_{T}\!\left(  \mathfrak{A}_{T},\mathfrak{B}_{T}\right)  ,\gamma
_{T}\!\left(  \mathfrak{A}_{T},\mathfrak{B}_{T}\right)  \right)  & \text{if
}T\neq C_{2},C_{2}\times C_{2},\\
\left(  \alpha_{T}\!\left(  \mathfrak{A}_{T},\mathfrak{B}_{T},\mathfrak{D}%
_{T}\right)  ,\beta_{T}\!\left(  \mathfrak{A}_{T},\mathfrak{B}_{T}%
,\mathfrak{D}_{T}\right)  ,\gamma_{T}\!\left(  \mathfrak{A}_{T},\mathfrak{B}%
_{T},\mathfrak{D}_{T}\right)  \right)  & \text{if }T=C_{2},C_{2}\times C_{2}.
\end{array}
\right.
\]
By \cite[Section~2]{2020arXiv200101016B}, we conclude that $T\hookrightarrow F_{T}\!\left(\mathbb{Q}\right)  $. This establishes $(1)$.

We now claim that the minimal discriminant of $F_{T}$ is $u_{T}^{-12}\gamma_{T}$
where%
\[
u_{T}=\left\{
\begin{array}
[c]{cl}%
1 & \text{if }T=C_{5},C_{7},C_{8},C_{9},C_{12},C_{2}\times C_{4},C_{2}\times
C_{8},\\
2 & \text{if }T=C_{2},C_{6},C_{10},C_{2}\times C_{2},\\
16 & \text{if }T=C_{2}\times C_{6},\\
\left(  a^{3}-3a^{2}b-6ab^{2}-b^{3}\right)  ^{2} & \text{if }T=C_{3},\\
\left(  a-b\right)  \left(  a+b\right)  ^{3} & \text{if }T=C_{4}.
\end{array}
\right.
\]
For $T\neq C_{2}$, this is a consequence of \cite[Theorem 4.4]{2020arXiv200101016B} since the assumptions on $a$ and $b$ imply that $\mathfrak{A}_{T},\mathfrak{B}_{T},$ and $\mathfrak{D}_{T}$ satisfy the
necessary and sufficient conditions for the claimed $u_{T}$. 
For $T=C_{2}$, we note that $\mathfrak{D}_{T}$ is not necessarily squarefree. Write $\mathfrak{D}_{T}=k^{2}l$ for some squarefree integer~$l$. Then $F_{T}
\!\left(  a,b\right)  =E_{T}\!\left(  \mathfrak{A}_{T},\mathfrak{B}
_{T}k,l\right)  $ satisfies the assumptions of loc. cit. Since $v_{2}\!\left(  \mathfrak{B}_{T}\right)  \geq3$
and $\mathfrak{A}_{T}\equiv3\ \operatorname{mod}4$, we conclude that $u_{T}=2$
by loc. cit. This shows the claim and thus the invariants $c_4$ and $c_6$ associated to a global minimal model of $F_T$ are $u_{T}^{-4}\alpha_T$ and $u_{T}^{-6}\beta_{T}$, respectively. It follows that $\left(  2\right)  $ holds since $u_{T}^{-4}\alpha
_{T}=A_{T}$, $u_{T}^{-6}\beta_{T}=B_{T}$, and $u_{T}^{-12}\gamma_{T}=D_{T}$. See \cite[Verification\_of\_Proposition\_3\_1.ipynb]{GitHubGoodEC} for a verification of these equalities.

Next, suppose that $T\neq C_{5}$ and recall that $\gcd(a,b)=1$ with $a \equiv 0 \mod 6$. If $T=C_{10}$ (resp. $T=C_7$), then $a$ is divisible by $5$ (resp. $7$). Consequently, $F_{T}$ is semistable since the integers $\mathfrak{A}_{T},\mathfrak{B}_{T},$ and $\mathfrak{D}_{T}$ satisfy one of the conditions listed in \cite[Corollary 7.2]{2020arXiv200101016B}. In particular, $\gcd\!\left(  A_{T},D_{T}\right)
=1$. From the identity $A_{T}^{3}-1728D_{T}=B_{T}^{2}$, we deduce that $v_{p}\!\left(
\gcd\!\left(  A_{T}^{3},B_{T}^{2}\right)  \right) =0$ for each prime $p\geq 5$. Since $a$ is divisible by $6$, we have by inspection that $A_{T}\equiv1\ \operatorname{mod}6$ which shows that $\gcd\!\left(  A_{T}^3,B_{T}^2\right)  =1$.

For $T=C_{5}$, we have that $\mathfrak{A}_{T}+3\mathfrak{B}_{T}%
\equiv4b^{20}+1\ \operatorname{mod}5$. By Fermat's Little Theorem, we deduce
that $\mathfrak{A}_{T}+3\mathfrak{B}_{T}$ is divisible by $5$. Consequently,
$F_{T}$ has additive reduction at $5$ and is semistable at all primes $p\neq
5$ \cite[Theorem 7.1]{2020arXiv200101016B}. Thus, $\gcd\!\left(  A_{T},D_{T}\right)  $ divides a power of $5$. Next, observe that $A_{T}%
\equiv1+16b^{20}+11b^{40}+16b^{60}+b^{80}\ \operatorname{mod}25$. Since
$\varphi\!\left(  25\right)  =20$, where $\varphi$ is the Euler-totient
function, we have by Euler's Theorem that $A_{T}\equiv20\ \operatorname{mod}%
25$. In particular, $v_{5}\!\left(  A_{T}\right)  =1$. Now let $P=64b^{8}-8b^{4}-1$, $Q=4096b^{16}-1024b^{12}+256b^{8}-24b^{4}+1$, and
$R=4096b^{16}+1536b^{12}+256b^{8}+16b^{4}+1$. Then $\frac{D_T}{PQR}=2^{75}b^{100}$. Since $b=2^{n}$ for some nonnegative integer $n$, it follows that
$v_{5}\!\left(  D_{T}\right)  =v_{5}\!\left(  PQR\right)  $. Moreover, \cite[Verification\_of\_Proposition\_3\_1.ipynb]{GitHubGoodEC} verifies that
$P,Q,R\equiv5\ \operatorname{mod}25$ whenever $b$ is not divisible by $5$.
Consequently, $v_{5}\!\left(  D_{T}\right)  =3$ and thus $\gcd\!\left(
A_{T},D_{T}\right)  =5$. From the identity $A_{T}^{3}-1728D_{T}=B_{T}^{2}$, we
observe that $v_{5}\!\left(  B_{T}\right)  \geq2$ and thus $v_{5}\!\left(
\gcd\!\left(  A_{T}^{3},B_{T}^{2}\right)  \right)  =3$. Lastly, we note that $A_{T}\equiv1\ \operatorname{mod}6$ and thus $v_{p}\!\left(
\gcd\!\left(  A_{T}^{3},B_{T}^{2}\right)  \right)  =0$ for each prime $p\neq 5$. It follows that $\gcd\!\left(  A_{T}^3,B_{T}^2\right)  =125$, which concludes the proof.
\end{proof}

\begin{lemma}
\label{isogeny}The following elliptic curves are isogenous and non-isomorphic
over $%
\mathbb{Q}
$:

$\left(  1\right)  $ $F_{C_{1}},\ F_{C_{3}},$ and $F_{C_{9}}$;

$\left(  2\right)  $ $F_{C_{4}},\ F_{C_{6}},\ F_{C_{12}},\ $and $F_{C_{2}%
\times C_{6}}$;

$\left(  3\right)  $ $F_{C_{2}},\ F_{C_{8}},\ F_{C_{2}\times C_{2}}%
,\ F_{C_{2}\times C_{4}},\ F_{C_{2}\times C_{8}}$, and $E_{C_{4}}\!\left(
256a^{2}b^{2}\left(  a^{2}-b^{2}\right)  ^{2},-\left(  a^{2}+b^{2}\right)
^{4}\right)  $ where $E_{C_{4}}$ is as defined in \cite[Table 3]%
{2020arXiv200101016B}.
\end{lemma}

\begin{proof}
Let $E_{n}=E_{n}\!\left(  s,t\right)  $ be as defined in \cite{MR1914669}. In
this work, Nitaj explicitly classified the isogeny class of an elliptic curve
having torsion subgroup $C_{9},\ C_{12},$ and $C_{2}\times C_{8}$ in terms of
the families of elliptic curves $E_{n}$. These families are related to $F_{T}$
by setting%
\[
\left(  \rho_{m},\tau_{m}\right)  =\left\{
\begin{array}
[c]{cl}%
\left(  a+b,a\right)   & \text{if }m=1,\\
\left(  b-2a,b\right)   & \text{if }m=2,\\
\left(  2a,-a-b\right)   & \text{if }m=3.
\end{array}
\right.
\]
Then $F_{C_{9}}=E_{45}\!\left(  \rho_{1},\tau_{1}\right)  $ and $F_{C_{1}%
}=E_{47}\!\left(  \rho_{1},\tau_{1}\right)  $. Verification of these equalities and the following isomorphisms can be found in \cite[Verification\_of\_Lemma\_3\_2.ipynb]{GitHubGoodEC}. It is then verified via
SageMath \cite{sagemath} that $F_{C_{3}}$ is $\mathbb{Q}\!\left(  a,b\right)  $-isomorphic to $E_{46}\!\left(  \rho_{1},\tau
_{1}\right)  $. Thus, $F_{C_{1}},\ F_{C_{3}},$ and $F_{C_{9}}$ are isogenous
and non-isomorphic by \cite{MR1914669}. The
following isomorphisms also hold over $
\mathbb{Q}
\!\left(  a,b\right)  $: $\left(  i\right)  $ $F_{C_{12}}\cong E_{52}\!\left(
\rho_{2},\tau_{2}\right)  $, $\left(  ii\right)  \ F_{C_{26}}\cong
E_{53}\!\left(  \rho_{2},\tau_{2}\right)  $, $\left(  iii\right)  $ $F_{C_{4}%
}\cong E_{54}\!\left(  \rho_{2},\tau_{2}\right)  $, $\left(  iv\right)
\ F_{C_{6}}\cong E_{55}\!\left(  \rho_{2},\tau_{2}\right)  $, $\left(
v\right)  \ F_{C_{2}\times C_8}\cong E_{37}\!\left(  \rho_{3},\tau_{3}\right)  $,
$\left(  vi\right)  $ $F_{C_{2}\times C_{4}}\cong E_{38}\!\left(  \rho
_{3},\tau_{3}\right)  $, $\left(  vii\right)  \ F_{C_{8}}\cong E_{39}\!\left(
\rho_{3},\tau_{3}\right)  $, $\left(  viii\right)  $ $F_{C_{2}\times C_{2}%
}\cong E_{41}\!\left(  \rho_{3},\tau_{3}\right)  $, 
$\left(  ix\right)  $ $F_{C_{2}}\cong E_{44}\!\left(  \rho_{3},\tau_{3}\right)
$, and 
$\left(  x\right)$ $E_{C_{4}}(  256a^{2}b^{2}(  a^{2}-b^{2})  ^{2},-(a^{2}+b^{2})  ^{4})  \cong E_{42}\!\left(  \rho_{3},\tau
_{3}\right)  $. The claim now follows by~\cite{MR1914669}.
\end{proof}

\begin{theorem}
\label{torsiso}For $T=C_{5}$, let $b=2^{n}$ for $n$ a nonnegative integer. For
$T\neq C_{5}$, let $a$ and $b$ be relatively prime integers with $a$ divisible
by $6$. Then $F_{T}\!\left(
\mathbb{Q}
\right)  _{\text{tors}}\cong T$.
\end{theorem}

\begin{proof}
By Proposition \ref{Propmin}, $T\hookrightarrow F_{T}\!\left(
\mathbb{Q}
\right)  $. Consequently, Theorem \ref{torsiontheorem} implies that the
theorem holds for $T=C_{7},C_{9},C_{10},C_{12},C_{2}\times C_{6},$ and
$C_{2}\times C_{8}$. By Lemma \ref{isogeny}, the following elliptic curves are
isogenous non-isomorphic elliptic curves over $%
\mathbb{Q}
$: $\left(  i\right)  \ F_{C_{1}}$, $F_{C_{3}},$ and $F_{C_{9}}$, $\left(
ii\right)  $ $F_{C_{4}},\ F_{C_{6}},\ F_{C_{12}},\ $and $F_{C_{2}\times C_{6}%
}$, and $\left(  iii\right)  $ $F_{C_{2}},\ F_{C_{8}},\ F_{C_{2}\times C_{2}%
},\ F_{C_{2}\times C_{4}},\ F_{C_{2}\times C_{8}}$. By \cite[Theorem
1.3]{2020arXiv200105616C}, there is exactly one isogeny-torsion graph for
$F_{C_{9}}$, $F_{C_{12}},$ and $F_{C_{2}\times C_{8}}$. Thus $F_{T}\!\left(
\mathbb{Q}
\right)  _{\text{tors}}\cong T$ if $T=C_{1}$,$\ C_{3}$,$\ C_{4}$,$C_{6}%
,C_{8},C_{2}\times C_{2},$ and $C_{2}\times C_{4}$. Since $F_{C_{2}\times
C_{8}}$ is isogenous to $E_{C_{4}}\!\left(  256a^{2}b^{2}\left(  a^{2}%
-b^{2}\right)  ^{2},-\left(  a^{2}+b^{2}\right)  ^{4}\right)  $ which has a
$4$-torsion point \cite[Section 2]{2020arXiv200101016B}, we conclude
by~\cite[Theorem 1.3]{2020arXiv200105616C}, that $F_{C_{2}}\!\left(
\mathbb{Q}
\right)  _{\text{tors}}\cong C_{2}$.

It remains to show the case when $T=C_{5}$. By assumption, $b=2^{n}$ for $n$ a
nonnegative integer. By Theorem \ref{torsiontheorem}, $F_{T}\!\left(
\mathbb{Q}
\right)  _{\text{tors}}$ is either $C_{5}$ or $C_{10}$. In particular,
$F_{T}\!\left(
\mathbb{Q}
\right)  _{\text{tors}}\cong C_{10}$ if and only if the $2$-division
polynomial%
\[
\psi_{2}\!\left(  x\right)  =4x^{3}+\left(  2^{30+40n}-2^{16+20n}3+1\right)
x^{2}+\left(  2^{31+40n}-2^{16+20n}\right)  x+2^{30+40n}%
\]
of $F_{T}$ has a root in $%
\mathbb{Q}
$. By the Rational Root Theorem, if $\psi_{2}\!\left(  x\right)  =0$ for $x\in%
\mathbb{Q}
$, then $x=\frac{1}{2},\ \frac{1}{4},$ or $2^{m}$ with $m\leq30+40n$. It
follows $\psi_{2}\!\left(  x\right)  $ has no rational solutions and thus
$F_{T}\!\left(
\mathbb{Q}
\right)  _{\text{tors}}\cong T$, which concludes the proof.
\end{proof}

\section{Sequences of Good \texorpdfstring{$ABC$}{}
Triples\label{sec:ABCtriples}}

In this section, we utilize the identity $A_{T}^{3}-B_{T}^{2}=1728D_{T}$ to
construct sequences of good $ABC$ triples for $T\neq C_{1},C_{2}$. In what
follows, let $\hat{D}_{T}=\hat{D}_{T}\!\left(  a,b\right)  $ be as defined in
Table \ref{ta:DhatT}. We start by establishing some technical results. These results have been verified in Mathematica \cite{Mathematica} and SageMath \cite{sagemath}, and the files containing the verifications are found in \cite{GitHubGoodEC}. Our first result is a consequence from the fact that
$A_{T},B_{T},D_{T},$ and $\hat{D}_{T}$ are homogeneous polynomials in $a$ and
$b$ if $T\neq C_{5}$.

\begin{lemma}
\label{HomPoly}For $T\neq C_{5}$, let%
\begin{equation}
n_{T}=\left\{
\begin{array}
[c]{cl}%
24 & \text{if }T=C_{7},\\
36 & \text{if }T=C_{1},C_{3},C_{9},C_{10},\\
48 & \text{if }T=C_{2},C_{4},C_{6},C_{8},C_{12},C_{2}\times C_{2},C_{2}\times
C_{4},C_{2}\times C_{6},C_{2}\times C_{8}.
\end{array}
\right.  \label{ch:gec:defnt}%
\end{equation}
Then%
\begin{align*}
A_{T}\!\left(  a,b\right)   &  =a^{\frac{n_{T}}{3}}A_{T}\!\left(  1,\frac
{b}{a}\right)  ,\qquad B_{T}\!\left(  a,b\right)  =a^{\frac{n_{T}}{2}}%
B_{T}\!\left(  1,\frac{b}{a}\right) \\
D_{T}\!\left(  a,b\right)   &  =a^{n_{T}}D_{T}\!\left(  1,\frac{b}{a}\right)
,\qquad\hat{D}_{T}\left(  a,b\right)  =a^{\frac{n_{T}}{6}}\hat{D}_{T}\!\left(
1,\frac{b}{a}\right)  .
\end{align*}

\end{lemma}

\begin{proof}
See \cite[Verification\_of\_Lemma\_4\_1.ipynb]{GitHubGoodEC} for a verification of this result.
\end{proof}

\begin{lemma}
\label{lemforhT}For each $T$, let $h_{T}:%
\mathbb{R}
\rightarrow%
\mathbb{R}
$ be defined by $h_{T}\!\left(  x\right)  =A_{T}\!\left(  1,x\right)
^{3}-\hat{D}_{T}\!\left(  1,x\right)  ^{6}$. Next, let $\theta_{T}%
=\max\left\{  x\in%
\mathbb{R}
\mid D_{T}\!\left(  1,x\right)  =0\right\}  $ so that%
\[
\theta_{T}\approx\left\{
\begin{array}
[c]{cl}%
4.41147 & \text{if }T=C_{1},C_{3},C_{9},\\
2.41421 & \text{if }T=C_{2},C_{8},C_{2}\times C_{2},C_{2}\times C_{4}%
,C_{2}\times C_{8},\\
3.73205 & \text{if }T=C_{4},C_{6},C_{12},C_{2}\times C_{6},\\
0.67062 & \text{if }T=C_{5},\\
6.2959 & \text{if }T=C_{7},\\
1.61803 & \text{if }T=C_{10}.
\end{array}
\right.
\]
Then the functions $A_{T}\!\left(  1,x\right)  ,\ D_{T}\!\left(  1,x\right)
,\ \hat{D}_{T}\!\left(  1,x\right)  $, and $h_{T}\!\left(  x\right)  $ are
positive on the interval $\left(  \theta_{T},\infty\right)  $.
\end{lemma}
\begin{proof}
For each $T$, \cite[Mathematica\_Verifications.nb]{GitHubGoodEC} verifies that $(1)$ $\theta_T$ is the greatest real root of $D_T(1,x)$, $(2)$ if $A_T(1,x),\hat{D}_T(1,x),$ or $h_T(x)$ has a real root, then the greatest real root is at most $\theta_T$, and $(3)$ $A_T(1,y),$ $D_T(1,y),$ $\hat{D}_T(1,y),$  $h_T(y)>0$ for some $y>\theta_T$. It follows that each of these functions are positive on the interval $(\theta_T,\infty)$ since they are continuous.
\end{proof}

\begin{lemma}
\label{poslemma}Let $T\neq C_{1},C_{2},C_{5}$ and let $\theta_T$ be as given in Lemma~\ref{lemforhT}. Define the functions
$f_{T},g_{T}:(\theta_T,\infty)
\rightarrow%
\mathbb{R}
$ by%
\[
f_{T}\!\left(  x\right)  =\frac{B_{T}\!\left(  1,x\right)  ^{2}}%
{1728D_{T}\!\left(  1,x\right)  }-x\qquad\text{and}\qquad g_{T}\!\left(
x\right)  =A_{T}\!\left(  1,x\right)  ^{2}+B_{T}\!\left(  1,x\right)  \hat
{D}_{T}\!\left(  1,x\right)  .
\]
Let $\delta_{T}=\max\left\{  x\in%
\mathbb{R}
\mid f_{T}\!\left(  x\right)  =0\right\}  $ so that%
\[
\delta_{T}\approx\left\{
\begin{array}
[c]{ll}%
43.4033 & \text{if }T=C_{3},\\
13.5934 & \text{if }T=C_{4},\\
43.3677 & \text{if }T=C_{6},\\
7.07956 & \text{if }T=C_{7},\\
2.48383 & \text{if }T=C_{8},\\
4.75552 & \text{if }T=C_{9},
\end{array}
\right.  \qquad\text{and}\qquad\delta_{T}\approx\left\{
\begin{array}
[c]{ll}%
3.06311 & \text{if }T=C_{10},\\
3.89418 & \text{if }T=C_{12},\\
1728.57 & \text{if }T=C_{2}\times C_{2},\\
12.2907 & \text{if }T=C_{2}\times C_{4},\\
6.00485 & \text{if }T=C_{2}\times C_{6},\\
3.38169 & \text{if }T=C_{2}\times C_{8}.
\end{array}
\right.
\]
Then the functions $f_{T}\!\left(  x\right)  ,\ g_{T}\!\left(  x\right)
,\ -B_{T}\!\left(  1,x\right)  $ are positive on the interval $\left(
\delta_{T},\infty\right)  $.
\end{lemma}
\begin{proof}
For $T\neq C_{1},C_{2},C_{5}$, \cite[Mathematica\_Verifications.nb]{GitHubGoodEC} verifies that $(1)$ $\delta_T$ is the greatest real root of $f_T(x)$, $(2)$ if $g_T(x)$ or $B_T(1,x)$ has a real root, then the greatest real root is at most $\delta_T$, and $(3)$ $f_T(y),$ $g_T(y),$ $-B_T(1,y)>0$ for some $y>\delta_T$. It follows that each of these functions are positive on the interval $(\delta_T,\infty)$ since they are continuous.
\end{proof}

\begin{corollary}
\label{corhT}Let $a$ and $b$ be relatively prime integers such that $\frac
{b}{a}>\theta_{T}$. For $T=C_{5}$, assume further that $a=n_{T}=1$. Then
$A_{T}\!\left(  a,b\right)  ,\ D_{T}\!\left(  a,b\right)  >0$ and $\hat{D}%
_{T}\!\left(  a,b\right)  ^{6}<A_{T}\!\left(  a,b\right)  ^{3}$.
\end{corollary}

\begin{proof}
Since $\frac{b}{a}>\theta_{T}$, we have by Lemmas \ref{HomPoly} and
\ref{lemforhT} that $A_{T}\!\left(  a,b\right)  =a^{\frac{n_{T}}{3}}%
A_{T}\!\left(  1,\frac{b}{a}\right)  >0$, $\ D_{T}\!\left(  a,b\right)
=a^{n_{T}}D_{T}\!\left(  1,\frac{b}{a}\right)  >0$, and
\[
A_{T}\!\left(  a,b\right)  ^{3}-\hat{D}\left(  a,b\right)  ^{6}=a^{n_{T}%
}\left(  A_{T}\!\left(  1,\frac{b}{a}\right)  ^{3}-\hat{D}_{T}\!\left(
1,\frac{b}{a}\right)  \right)  >0. \qedhere
\] 
\end{proof}

\begin{corollary}
\label{corfT}Let $T\neq C_{1},C_{2},C_{5}$. Let $a$ and $b$ be relatively
prime integers such that $\frac{b}{a}>\delta_{T}$. Then%
\[
\frac{b}{a}<\frac{B_{T}\!\left(  a,b\right)  ^{2}}{1728D_{T}\!\left(
a,b\right)  }\qquad\text{and}\qquad\left\vert B_{T}\!\left(  a,b\right)
\hat{D}_{T}\!\left(  a,b\right)  \right\vert <A_{T}\!\left(  a,b\right)
^{2}.
\]

\end{corollary}

\begin{proof}
Since $\frac{b}{a}>\delta_{T}$, Lemmas \ref{HomPoly} and \ref{poslemma} imply
that%
\begin{align*}
\frac{B_{T}\!\left(  a,b\right)  ^{2}}{1728D_{T}\!\left(  a,b\right)  }%
-\frac{b}{a}  &  =\frac{B_{T}\!\left(  1,\frac{b}{a}\right)  ^{2}}%
{1728D_{T}\!\left(  1,\frac{b}{a}\right)  }-\frac{b}{a}>0\qquad\text{and}\\
A_{T}\!\left(  a,b\right)  ^{2}-\left\vert B_{T}\!\left(  a,b\right)  \hat
{D}_{T}\!\left(  a,b\right)  \right\vert  &  =a^{\frac{2n_{T}}{3}}\left(
A_{T}\!\left(  1,\frac{b}{a}\right)  ^{2}-\left\vert B_{T}\!\left(  1,\frac
{b}{a}\right)  \hat{D}_{T}\!\left(  1,\frac{b}{a}\right)  \right\vert \right)
\\
&  =a^{\frac{2n_{T}}{3}}\left(  A_{T}\!\left(  1,\frac{b}{a}\right)
^{2}+B_{T}\!\left(  1,\frac{b}{a}\right)  \hat{D}_{T}\!\left(  1,\frac{b}%
{a}\right)  \right)  >0.
\end{align*}
The result now follows.
\end{proof}

\begin{lemma}
\label{lemABC}For $T=C_{5}$, let $b=2^{n}$ for $n$ a nonnegative integer. For
$T\neq C_{5}$, let $a$ and $b$ be relatively prime integers with $a$ divisible
by $6$ such that $\frac{b}{a}>\theta_{T}$ where $\theta_{T}$ is as defined in
Lemma~\ref{lemforhT}. Suppose further that $a\equiv0\ \operatorname{mod}5$
(resp. $a\equiv0\ \operatorname{mod}7$) if $T=C_{10}$ (resp. $T=C_{7}$). Then
$\left(  1728D_{T}\!\left(  a,b\right)  ,B_{T}^{2}\!\left(  a,b\right)
,A_{T}^{3}\!\left(  a,b\right)  \right)  $ and $\left(  5^{-3}1728D_{T}%
\!\left(  1,b\right)  ,5^{-3}B_{T}^{2}\!\left(  1,b\right)  ,5^{-3}A_{T}%
^{3}\!\left(  1,b\right)  \right)  $ are positive $ABC$ triples for $T\neq
C_{5}$ and $T=C_{5}$, respectively. In particular, $\max\!\left\{  \left\vert
A_{T}^{3}\right\vert ,B_{T}^{2}\right\}  =A_{T}^{3}$.
\end{lemma}

\begin{proof}
From Proposition \ref{Propmin}, we have the identity $1728D_{T}+B_{T}^{2}%
=A^3_T\ $and%
\[
\gcd\!\left(  A_{T}^{3},B_{T}^{2}\right)  =\left\{
\begin{array}
[c]{cl}%
1 & \text{if }T\neq C_{5},\\
125 & \text{if }T=C_{5}.
\end{array}
\right.
\]
Consequently, $\left(  1728D_{T},B_{T}^{2},A_{T}^{3}\right)  $ and $\left(
5^{-3}1728D_{T},5^{-3}B_{T}^{2},5^{-3}A_{T}^{3}\right)  $ are $ABC$ triples
for $T\neq C_{5}$ and $T=C_{5}$, respectively. Moreover, these are positive
$ABC$ triples by Corollary \ref{corhT}.
\end{proof}

\begin{lemma}
\label{lemmadisc}Let $T\neq C_{5}$. If $\left(  a,b,a+b\right)  $ is a good
positive $ABC$ triple with $a$ even, then $\operatorname{rad}\!\left(
D_{T}\!\left(  a,b\right)  \right)  <\left\vert \hat{D}_{T}\!\left(
a,b\right)  \right\vert $.
\end{lemma}

\begin{proof}
Since $\left(  a,b,a+b\right)  $ is a good positive $ABC$ triple,
$\operatorname{rad}\!\left(  ab\left(  a+b\right)  \right)  <a+b$. In
particular, for any positive integers $d$ and $k$, it is the case that
$\operatorname{rad}\!\left(  ab\left(  a+b\right)  d^{k}\right)  <\left(
a+b\right)  d$. Moreover, if $2^{k}$ divides $a$, then $\operatorname{rad}%
\!\left(  \frac{a}{2^{k}}\right)  \leq\operatorname{rad}\!\left(  a\right)  $.
Using these two statements it is easy to verify by inspection that the claim
holds for all $T$ with the possible exception of $T=C_{10}$. For $T=C_{10}$,
observe that $P=\left(  a+2b\right)  \left(  a^{2}+6ab+4b^{2}\right)  \left(
-a^{2}-ab+b^{2}\right)  $ is divisible by $8$ since $a$ is even. Thus $P=8Q$
for some integer $Q$ and%
\[
\operatorname{rad}\!\left(  D_{T}\!\left(  a,b\right)  \right)  \leq
\operatorname{rad}\!\left(  8ab\left(  a+b\right)  Q\right)
=\operatorname{rad}\!\left(  ab\left(  a+b\right)  Q\right)  <\left(
a+b\right)  \operatorname{rad}\!\left(  Q\right)  .
\]
The claim now follows since $\left(  a+b\right)  \operatorname{rad}\!\left(
Q\right)  \leq\left(  a+b\right)  \left\vert \frac{P}{8}\right\vert
=\left\vert \hat{D}_{T}\!\left(  a,b\right)  \right\vert $.
\end{proof}

\begin{theorem}
\label{excepseq}Let $T\neq C_{1},C_{2},C_{5}$. Let $P_{0}^{T}=\left(
a_{0},b_{0},c_{0}\right)  $ be a good positive $ABC$ triple with $a_{0}%
\equiv0\ \operatorname{mod}6$ and $\frac{b_{0}}{a_{0}}>\delta_{T}$ where
$\delta_{T}$ is as defined in Lemma \ref{poslemma}. Define $P_{n}^{T}$ for
$n\geq1$ recursively by
\[
P_{n}^{T}=\left(  a_{n},b_{n},c_{n}\right)  =\left(  1728D_{T}\!\left(
a_{n-1},b_{n-1}\right)  ,B_{T}\!\left(  a_{n-1},b_{n-1}\right)  ^{2}%
,A_{T}\!\left(  a_{n-1},b_{n-1}\right)  ^{3}\right)  .
\]
Assume further that $a_{0}\equiv0\ \operatorname{mod}5$ (resp. $a_{0}%
\equiv0\ \operatorname{mod}7$) whenever $T=C_{10}$ (resp. $T=C_{7}$). Then the
following hold for each $n\geq0$.

$\left(  1\right)  $ $\frac{b_{n}}{a_{n}}>\delta_{T}$;

$\left(  2\right)  \ a_{n}\equiv0\ \operatorname{mod}6$ for $n\geq0$;

$\left(  3\right)  $ $a_{n}\equiv0\ \operatorname{mod}5$ (resp. $a_{n}%
\equiv0\ \operatorname{mod}7$) whenever $T=C_{10}$ (resp. $T=C_{7}$);

$\left(  4\right)  $ $P_{n}^{T}=\left(  a_{n},b_{n},c_{n}\right)  $ is a good
positive $ABC$ triple.
\end{theorem}

\begin{proof}
By Corollary \ref{corfT}, we have that $\frac{b_{0}}{a_{0}}<\frac{b_{1}}%
{a_{1}}$. By induction on $n$, we deduce that $\left\{  \frac{b_{n}}{a_{n}%
}\right\}  $ is an increasing sequence of rational numbers. This shows
$\left(  1\right)  $. By definition of $D_{T}$, we see that $a_{0}$ divides
$a_{n}$ for each $n\geq0$. Thus $\left(  2\right)  $ and $\left(  3\right)  $
hold. By Lemma \ref{lemABC}, we conclude that $P_{n}^{T}$ is a positive $ABC$
triple since $\delta_{T}>\theta_{T}$.

We now show that $P_{n}^{T}$ is a good $ABC$ triple by induction on $n$. Since
the base case is given, we may assume $P_{n}^{T}=\left(  a_{n},b_{n}%
,c_{n}\right)  $ is a good positive $ABC$ triple for some $n$. By Lemma
\ref{lemmadisc},%
\begin{align*}
\operatorname{rad}\!\left(  a_{n+1}b_{n+1}c_{n+1}\right)   &
=\operatorname{rad}\!\left(  D_{T}\!\left(  a_{n},b_{n}\right)  B_{T}\!\left(
a_{n},b_{n}\right)  A_{T}\!\left(  a_{n},b_{n}\right)  \right) \\
&  <\left\vert \hat{D}_{T}\!\left(  a_{n},b_{n}\right)  B_{T}\!\left(
a_{n},b_{n}\right)  \right\vert A_{T}\!\left(  a_{n},b_{n}\right)  .
\end{align*}
Next, observe that
\[
c_{n+1}-\left\vert \hat{D}_{T}\!\left(  a_{n},b_{n}\right)  B_{T}\!\left(
a_{n},b_{n}\right)  \right\vert A_{T}\!\left(  a_{n},b_{n}\right)
=A_{T}\!\left(  a_{n},b_{n}\right)  \left(  A_{T}\!\left(  a_{n},b_{n}\right)
^{2}-\left\vert \hat{D}_{T}\!\left(  a_{n},b_{n}\right)  B_{T}\!\left(
a_{n},b_{n}\right)  \right\vert \right)
\]
is positive by Corollary \ref{corfT}. Consequently, $\operatorname{rad}%
\!\left(  a_{n+1}b_{n+1}c_{n+1}\right)  <c_{n+1}$ and thus $P_{n+1}^{T}$ is a
good positive $ABC$ triple, which completes the proof.
\end{proof}

Next, we show that there are good $ABC$ triples which satisfy the assumptions
of Theorem~\ref{excepseq}. To this end, let $P_{0}^{T}=\left(  a_{0}%
,b_{0},c_{0}\right)  $ be the good $ABC$ triple given below:%
\begin{equation}
{\renewcommand*{\arraystretch}{1.35}%
\begin{tabular}
[c]{cccccc}\hline
$T$ & $a_{0}$ & $b_{0}$ & $c_{0}$ & $\frac{b_{0}}{a_{0}}$ & $q\!\left(
P_{0}^{T}\right)  $\\\hline
$C_{8},C_{9},C_{12},C_{2}\times C_{6},C_{2}\times C_{8}$ & $2^{6}3$ & $47^{2}$
& $7^{4}$ & $\approx11.5$ & $\approx1.0258$\\\hline
$C_{3},C_{4},C_{6},C_{2}\times C_{4}$ & $2\cdot3^{4}$ & $5\cdot7^{4}$ &
$23^{3}$ & $\approx74.1$ & $\approx1.1089$\\\hline
$C_{7}$ & $2^{6}3^{3}7$ & $5\cdot31^{3}$ & $11^{5}$ & $\approx12.3$ &
$\approx1.0725$\\\hline
$C_{10}$ & $2^{6}3^{2}5$ & $13\cdot23^{3}$ & $11^{5}$ & $\approx54.9$ &
$\approx1.0426$\\\hline
$C_{2}\times C_{2}$ & $2^{4}3^{2}$ & $5^{6}53$ & $13^{4}29$ & $\approx5750.9$
& $\approx1.0243$\\\hline
\end{tabular}
} \label{TableP0T}%
\end{equation}
Therefore, Theorem \ref{excepseq} implies that the corresponding sequence
$P_{n}^{T}=\left(  a_{n},b_{n},c_{n}\right)  $ is a sequence of good positive
$ABC$ triples for each $n$. Below we list an approximation of the quality of
the good $ABC$ triple $  P_{1}^{T}  $. We do not include
$a_{1}$ or $b_{1}$ due to their size. For instance, $a_{1}$ and $b_{1}$ of
$P_{1}^{C_{3}}$ have $143$ and $145$ digits, respectively. Consequently,
computing the quality of $P_{n}^{T}$ for $n\geq2$ is not computationally
feasible. We note that the calculations for the approximate values found in \eqref{TableP0T} and \eqref{Pqualities} are found in \cite[Qualities\_and\_MSRs.ipynb]{GitHubGoodEC}.
\begin{equation}
{\renewcommand*{\arraystretch}{1.2}%
\begin{tabular}
[c]{ccccccc}\hline
$T$ & $C_{3}$ & $C_{4}$ & $C_{6}$ & $C_{7}$ & $C_{8}$ & $C_{9}$\\\hline
$q\!\left(  P_{1}^{T}\right)  \approx$ & $1.00279$ & $1.00209$ & $1.00209$ &
$1.00381$ & $1.00057$ & $1.00119$\\\hline
&&&&&&\\\hline
$T$ & $C_{10}$ & $C_{12}$ & $C_{2}\times C_{2}$ & $C_{2}\times C_{4}$ &
$C_{2}\times C_{6}$ & $C_{2}\times C_{8}$\\\hline
$q\!\left(  P_{1}^{T}\right) \approx $ & $1.00803$ & $1.00084$ & $1.00049$ &
$1.00205$ & $1.00084$ & $1.00057$\\\hline
\end{tabular}
} \label{Pqualities}
\end{equation}

\begin{lemma}
\label{CondC5}Let $T=C_{5}$ and set $b=2^{n}$ for $n$ a nonnegative integer.
Then $5\operatorname{rad}\!\left(  D_{T}\right)  \leq\left\vert \hat{D}%
_{T}\right\vert $.
\end{lemma}

\begin{proof}
Observe that $D_T$ has a factor of $2^{75}b^{100}$.
By assumption, $\operatorname{rad}\!\left(  2^{75}b^{100}\right)  =2$ and thus
$\operatorname{rad}\!\left(  D_{T}\right)  =\operatorname{rad}\!\left(
5\hat{D}_{T}\right)  $. Next, observe that each factor of $\frac{5}{2}\hat
{D}_{T}$ is divisible by $5$ by Fermat's Little Theorem. Consequently, $25$
divides $\hat{D}_{T}$. In particular, $\operatorname{rad}\!\left(  5\hat
{D}_{T}\right)  =\operatorname{rad}\!\left(  \frac{\hat{D}_{T}}{5}\right)  $
and hence $5\operatorname{rad}\!\left(  D_{T}\right)  \leq\left\vert \hat
{D}_{T}\right\vert $.
\end{proof}

\begin{corollary}
Let $T=C_{5}$. Then
\[
P_{n}^{T}=\left(  a_{n},b_{n},c_{n}\right)  =\left(  5^{-3}1728D_{T}\!\left(
1,2^{n}\right)  ,5^{-3}B_{T}^{2}\!\left(  1,2^{n}\right)  ,5^{-3}A_{T}%
^{3}\!\left(  1,2^{n}\right)  \right)
\]
is a good positive $ABC$ triple for each $n\geq0$.
\end{corollary}

\begin{proof}
By Lemma \ref{lemABC}, $\left(  a_{n},b_{n},c_{n}\right)  $ is a positive
$ABC$ triple for each $n$. By the proof of Proposition~\ref{Propmin},
$v_{5}\!\left(  5^{-3}B_{T}^{2}\right)  \geq1$. This, alongside Lemma
\ref{CondC5}, imply that%
\begin{align*}
\operatorname{rad}\!\left(  a_{n}b_{n}c_{n}\right)   &  =\operatorname{rad}%
\!\left(  3D_{T}\!\left(  1,2^{n}\right)  B_{T}\!\left(  1,2^{n}\right)
A_{T}\!\left(  1,2^{n}\right)  \right) \\
&  <\frac{3}{5}\left\vert \hat{D}\!\left(  1,2^{n}\right)  B_{T}\!\left(
1,2^{n}\right)  \right\vert A_{T}\!\left(  1,2^{n}\right)  .
\end{align*}
It therefore suffices to show that%
\begin{equation}
A_{T}\!\left(  1,2^{n}\right)  ^{3}-\frac{3}{5}\left\vert \hat{D}\!\left(
1,2^{n}\right)  B_{T}\!\left(  1,2^{n}\right)  \right\vert A_{T}\!\left(
1,2^{n}\right)  >0. \label{radC5}%
\end{equation}
By \cite[Mathematica\_Verifications.nb]{GitHubGoodEC}, we have that the function
$A_{T}\!\left(  1,x\right)  ^{2}-\frac{3}{5}\left\vert \hat{D}\!\left(
1,x\right)  B_{T}\!\left(  1,x\right)  \right\vert$
is positive on the interval $[1,\infty)$ which shows that (\ref{radC5}) holds
for each $n\geq0$.
\end{proof}

Below, we list the value $q\!\left(  P_{n}^{C_{5}}\right)  $ (rounded to five decimal places)
for $n\leq4$. 

\begin{equation}
{\renewcommand*{\arraystretch}{1.2}%
\begin{tabular}
[c]{cccccc}\hline
$n$ & $0$ & $1$ & $2$ & $3$ & $4$\\\hline
$q\!\left(  P_{n}^{C_{5}}\right)  $ & $1.01204$ & $1.00501$ & $1.00316$ &
$1.00493$ & $1.00182$\\\hline
\end{tabular}
}
\end{equation}

\section{Infinitely Many Good Elliptic Curves\label{InfManyGood}}

In this section, we show that for each $T$ allowed by Theorem
\ref{torsiontheorem}, there are infinitely many good elliptic curves $E$ with
$E\!\left(
\mathbb{Q}
\right)  _{\text{tors}}\cong T$. Recall that an elliptic curve $E$ is good if
$N_{E}^{6}<\max\!\left\{  \left\vert c_{4}^{3}\right\vert ,c_{6}^{2}\right\}
$. Now let $N_{T}=N_{F_{T}\!\left(  a,b\right)  }$ denote the conductor of
$F_{T}\!\left(  a,b\right)  $. Under the assumptions of Proposition
\ref{Propmin}, showing that $F_{T}$ is good is equivalent to showing that the
inequality $N_{T}^{6}<\max\!\left\{  \left\vert A_{T}^{3}\right\vert
,B_{T}^{2}\right\}  $ holds since $A_{T}$ and $B_{T}$ are the invariants
$c_{4}$ and $c_{6}$ associated to a global minimal model of $F_{T}$, respectively.

Next, let $P_{0}^{T}$ be as given in (\ref{TableP0T}) for $T\neq C_{1}%
,C_{2},C_{5}$. Then Theorem \ref{excepseq} implies that the $ABC$ triple
$P_{n}^{T}=\left(  a_{n},b_{n},c_{n}\right)  $ is good for $n\geq0$. Moreover,
$\frac{b_{n}}{a_{n}}>\delta_{T}>\theta_{T}$. For $T=C_{1},C_{2}$, let
$P_{n}^{T}=P_{n}^{C_{8}}$. With these sequences of good $ABC$ triples, the
theorem below leads to the construction of infinitely many good elliptic
curves with a specified torsion subgroup.

\begin{theorem}
\label{goodtorsionmainthm}For $T\neq C_{5}$, suppose $\left\{  P_{n}%
^{T}\right\}  _{n\geq0}=\left\{  \left(  a_{n},b_{n},c_{n}\right)  \right\}
_{n\geq0}$ is a sequence of good positive $ABC$ triples such that $a_{n}%
\equiv0\ \operatorname{mod}6$ and $\frac{b_{n}}{a_{n}}>\theta_{T}$ where
$\theta_{T}$ is as defined in Lemma \ref{lemforhT}. Suppose further that
$a_{n}\equiv0\ \operatorname{mod}5$ (resp. $a_{n}\equiv0\ \operatorname{mod}%
7$) whenever $T=C_{10}$ (resp. $T=C_{7}$).

For $n\geq0$, let $H_{T}\!\left(  n\right)  $ be the elliptic curve%
\[
H_{T}\!\left(  n\right)  =\left\{
\begin{array}
[c]{cl}%
F_{T}\!\left(  a_{n},b_{n}\right)  & \text{if }T\neq C_{5},\\
F_{T}\!\left(  1,2^{n}\right)  & \text{if }T=C_{5}.
\end{array}
\right.
\]
Then, for each $n$, $H_{T}\!\left(  n\right)  $ is a good elliptic curve with
$H_{T}\!\left(  n\right)  \!\left(
\mathbb{Q}
\right)  _{\text{tors}}\cong T$.
\end{theorem}

\begin{proof}
For each $T$, Theorem \ref{torsiso} implies that $H_{T}\!\left(  n\right)
\!\left(
\mathbb{Q}
\right)  _{\text{tors}}\cong T$. Now suppose $T=C_{5}$. Then Proposition
\ref{Propmin} implies that the minimal discriminant of $H_{T}\!\left(
n\right)  $ is $D_{T}\!\left(  1,2^{n}\right)  $ for each $n\geq0$.
Moreover, the invariants $c_{4}$ and $c_{6}$ associated to a global minimal
model of $H_{T}\!\left(  n\right)  $ are $A_{T}\!\left(  1,2^{n}\right)  $
and $B_{T}\!\left(  1,2^{n}\right)  $, respectively. By Lemma \ref{lemABC},
$A_{T}\!\left(  1,2^{n}\right)  $ and $D_{T}\!\left(  1,2^{n}\right)  $
are positive and thus $\max\!\left\{  \left\vert c_{4}^{3}\right\vert
,c_{6}^{2}\right\}  =A_{T}\!\left(  1,2^{n}\right)  ^{3}$. By Proposition
\ref{Propmin}, $F_{T}$ has additive reduction at $5$ and is semistable at each
prime $p\neq5$. Consequently, $N_{T}=5\operatorname{rad}\!\left(
D_{T}\right)  $ since $D_{T}$ is the minimal discriminant of $E_{T}$. From
Lemma \ref{CondC5} we deduce that $N_{T}^{6}\leq\hat{D}_{T}^{6}$. By Corollary
\ref{corhT}, we conclude that $H_{T}\!\left(  n\right)  $ is a good elliptic
curve for each $n$ since%
\[
N_{T}\!\left(  1,2^{n}\right)  ^{6}\leq\hat{D}_{T}\!\left(  1,2^{n}%
\right)  ^{6}<A_{T}\!\left(  1,2^{n}\right)  ^{3}.
\]

Now suppose $T\neq C_{5}$. For each $n\geq0$, we have that $H_{T}\!\left(
n\right)  $ is semistable by Proposition~\ref{Propmin} and Theorem~\ref{excepseq}. Moreover, the minimal
discriminant of $H_{T}\!\left(  n\right)  $ is $D_{T}\!\left(  a_{n}%
,b_{n}\right)  $ and the invariants $c_{4}$ and $c_{6}$ associated to a global
minimal model of $H_{T}\!\left(  n\right)  $ are $A_{T}\!\left(  a_{n}%
,b_{n}\right)  $ and $B_{T}\!\left(  a_{n},b_{n}\right)  $, respectively. Thus
$N_{T}\!\left(  a_{n},b_{n}\right)  =\operatorname{rad}\!\left(
D_{T}\!\left(  a_{n},b_{n}\right)  \right)  $. By Lemma \ref{lemABC},
$\max\!\left\{  \left\vert A_{T}\!\left(  a_{n},b_{n}\right)  \right\vert
^{3},B_{T}^{2}\!\left(  a_{n},b_{n}\right)  \right\}  =A_{T}\!\left(
a_{n},b_{n}\right)  ^{3}$. Therefore, Corollary \ref{corhT} and Lemma
\ref{lemmadisc} imply that%
\[
N_{T}\!\left(  a_{n},b_{n}\right)  ^{6}<\hat{D}_{T}\!\left(  a_{n}%
,b_{n}\right)  ^{6}<A_{T}\!\left(  a_{n},b_{n}\right)  ^{3}.
\]
This shows that $H_{T}\!\left(  n\right)  $ is a good elliptic curve for each
$n$.
\end{proof}

For $T\neq C_{5}$, let $P_{n}^{T}=\left(  a_{n},b_{n},c_{n}\right)  $ be the
sequence of good $ABC$ triples defined in the discussion prior to Theorem
\ref{goodtorsionmainthm}. Applying Theorem \ref{goodtorsionmainthm} to these
sequences yields an infinite sequence of good elliptic curves $H_{T}\!\left(
n\right)  $. Below, we list the modified Szpiro ratio of $H_{T}\!\left(
0\right)  $ and note that we omit $\sigma_{m}\!\left(  H_{T}\!\left(
n\right)  \right)  $ for $n\geq1$ due to the difficulty in computing $N_{H_{T}\!\left(
n\right)  }$. We note that the calculations for the approximate values of $\sigma_{m}\!\left(  H_{T}\!\left(
0\right)  \right)  $ are found in \cite[Qualities\_and\_MSRs.ipynb]{GitHubGoodEC}.
\[
{\renewcommand*{\arraystretch}{1.2}%
\begin{tabular}
[c]{ccccccccc}\hline
$T$ & $C_{1}$ & $C_{2}$ & $C_{3}$ & $C_{4}$ & $C_{5}$ & $C_{6}$ & $C_{7}$ &
$C_{8}$\\\hline
$\sigma_{m}\!\left(  H_{T}\!\left( 0\right)  \right)  \approx$ & $6.3020$ &
$6.2060$ & $6.1016$ & $6.0759$ & $6.2766$ & $6.0759$ & $6.1398$ &
$6.0207$\\\hline
&  &  &  &  &  &  &  & \\\hline
$T$ & $C_{9}$ & $C_{10}$ & $C_{12}$ & $C_{2}\times C_{2}$ & $C_{2}\times
C_{4}$ & $C_{2}\times C_{6}$ & $C_{2}\times C_{8}$ & \\\hline
$\sigma_{m}\!\left(  H_{T}\!\left(  0\right)  \right) \approx $ & $6.04327$ &
$6.0412$ & $6.0304$ & $6.0179$ & $6.0747$ & $6.0304$ & $6.0207$ &
\\\hline
\end{tabular}
}%
\]

\begin{lemma}
\label{LemSS}Let $E$ be a good semistable elliptic curve with minimal
discriminant $\Delta\equiv0\ \operatorname{mod}6$. Let $c_{4}$ and $c_{6}$ be
the invariants associated to a global minimal model of $E$. Then, the $ABC$
triple $\left(  1728\Delta_{E},-c_{4}^{3},c_{6}^{2}\right)  $ is good.
\end{lemma}

\begin{proof}
Since $E$ is good, $N_{E}^{6}<\max\!\left\{  \left\vert c_{4}^{3}\right\vert
,c_{6}^{2}\right\}  $. Since $E$ is semistable and $\Delta\equiv
0\ \operatorname{mod}6$, we have that $\left(  1728\Delta_{E},-c_{4}^{3}%
,c_{6}^{2}\right)  $ is an $ABC$ triple. Moreover, $\operatorname{rad}%
\!\left(  1728\Delta_{E}\right)  =N_{E}$. Now observe that%
\[
\operatorname{rad}\!\left(  1728\Delta c_{4}c_{6}\right)  =N_{E}%
\operatorname{rad}\!\left(  c_{4}c_{6}\right)  \leq N_{E}\left\vert
c_{4}\right\vert \left\vert c_{6}\right\vert .
\]
It suffices to show that $N_{E}\left\vert c_{4}\right\vert \left\vert
c_{6}\right\vert <\max\!\left\{  \left\vert c_{4}^{3}\right\vert ,c_{6}%
^{2}\right\}  $ since this would imply that $\operatorname{rad}\!\left(
1728\Delta c_{4}c_{6}\right)  <\max\!\left\{  \left\vert c_{4}^{3}\right\vert
,c_{6}^{2},1728\left\vert \Delta_{E}\right\vert \right\}  $.

\textbf{Case 1.} Suppose $\max\!\left\{  \left\vert c_{4}^{3}\right\vert
,c_{6}^{2}\right\}  =\left\vert c_{4}^{3}\right\vert $. Then $N_{E}<\left\vert
c_{4}\right\vert ^{1/2}$ and $\left\vert c_{6}\right\vert <\left\vert
c_{4}\right\vert ^{3/2}$ since $N_{E}^{6},c_{6}^{2}<\left\vert c_{4}%
^{3}\right\vert $. Consequently, 
$N_{E}\left\vert c_{4}\right\vert \left\vert c_{6}\right\vert <\left\vert
c_{4}\right\vert ^{1/2}\left\vert c_{4}\right\vert \left\vert c_{4}\right\vert
^{3/2}=\left\vert c_{4}^{3}\right\vert$ .

\textbf{Case 2.} Suppose $\max\!\left\{  \left\vert c_{4}^{3}\right\vert
,c_{6}^{2}\right\}  =c_{6}^{2}$. Then $N_{E}<\left\vert c_{6}\right\vert
^{1/3}$ and $\left\vert c_{4}\right\vert <\left\vert c_{6}\right\vert ^{2/3}$.
Hence%
\[
N_{E}\left\vert c_{4}\right\vert \left\vert c_{6}\right\vert <\left\vert
c_{6}\right\vert ^{1/3}\left\vert c_{6}\right\vert ^{2/3}\left\vert
c_{6}\right\vert =c_{6}^{2}.
\]
In both cases we have that $\operatorname{rad}\!\left(  1728\Delta c_{4}%
c_{6}\right)  <\max\!\left\{  \left\vert c_{4}^{3}\right\vert ,c_{6}%
^{2},1728\left\vert \Delta_{E}\right\vert \right\}  $.
\end{proof}

The following result is immediate by Lemma \ref{LemSS}.

\begin{corollary}
Assume the statement of Theorem \ref{goodtorsionmainthm} and let $\left\{
H_{T}\!\left(  n\right)  \right\}  _{n\geq0}$ be the sequence of good elliptic
curves associated to $T\neq C_{5}$. Then
\[
\left\{  \left(  1728D_{T}\!\left(  a_{n},b_{n}\right)  ,B_{T}\!\left(
a_{n},b_{n}\right)  ^{2},A_{T}\!\left(  a_{n},b_{n}\right)
^{3}\right)  \right\}
\]
is a sequence of good $ABC$ triples for each $n\geq0$.
\end{corollary}

We conclude this section with some commentary on Lemma \ref{LemSS}. First, we
note that the assumption that $\Delta\equiv0\ \operatorname{mod}6$ is
necessary as illustrated by the following two examples which do not satisfy
this congruence.

\begin{example}
Let $E_{1}$ and $E_{2}$ be the elliptic curves given by the Weierstrass models%
\begin{align*}
E_{1}  &  :y^{2}+xy=x^{3}+x^{2}-2217690832x+40177507619839,\\
E_{2}  &  :y^{2}+xy+y=x^{3}-2091298661x+36762769880016.
\end{align*}
Both curves are semistable since $N_{E_{1}}=3\cdot5\cdot13\cdot19\cdot83$ and
$N_{E_{2}}=2\cdot7\cdot21799$. Moreover, the invariants $c_{4,j}$ and
$c_{6,j}$ associated to a global minimal model of $E_{j}$ are%
\[%
\begin{array}
[c]{llll}%
c_{4,1}=827\cdot128717243 & \qquad & \qquad & c_{6,2}=-17\cdot53\cdot
3023\cdot5471\cdot2329577,\\
c_{4,2}=5\cdot20076467141 & \qquad & \qquad & c_{6,2}=-5869\cdot5412026537713.
\end{array}
\]
Consequently, both curves are good since $N_{E_{j}}^{6}<c_{4,j}^{3}$ for
$j=1,2$. Next, let $\Delta_{E_{j}}$ denote the minimal discriminant of $E_{j}%
$. Since the primes dividing $\Delta_{E_{j}}$ are exactly the primes dividing
$N_{E_{j}}$, we deduce that $\left(  1728\Delta_{E_{j}},c_{6,j}^{2}%
,c_{4,j}^{3}\right)  $ is a positive $ABC$ triple. However, it is not good
since $\operatorname{rad}\!\left(  6\Delta_{E_{j}}c_{6,j}c_{4,j}\right)
>c_{4,j}^{3}$ for $j=1,2$.
\end{example}

In addition, the converse to Lemma \ref{LemSS} does not hold. This is
illustrated by considering the elliptic curve $E$ given by
\[
E:y^{2}+xy=x^{3}-2342114817x-46491207963039.
\]
The curve $E$ is semistable since $N_{E}=2\cdot3\cdot7\cdot67\cdot127$. Next,
let $c_{4}$ and $c_{6}$ be the invariants associated to a global minimal model
of $E$ and let $\Delta$ denote the minimal discriminant of $E$. Then%
\[
\Delta=-2^{3}3^{6}7^{3}67^{9}127^{3},\qquad c_{4}=19\cdot53\cdot
157\cdot251\cdot2833,\qquad c_{6}=13^{2}73\cdot5651\cdot576166333.
\]
Hence, $\left(  -1728\Delta,c_{4}^{3},c_{6}^{2}\right)  $ is a positive $ABC$
triple satisfying $\Delta\equiv0\ \operatorname{mod}6$. The $ABC$ triple is
also good since $\operatorname{rad}\!\left(  \Delta c_{6}c_{4}\right)
<c_{6}^{2}$. However, $E$ is not good since $c_{6}^{2}<N_{E}^{6}$. We note
that $E$ is $3$-isogenous to the good elliptic curve $F:y^{2}+xy=x^{3}%
-193149169647x-32672893402475361$.

\section{Good Elliptic Curves Arising from
\texorpdfstring{$F_{T}$}{}\label{sec:examples}}

In \cite{MR1273410} and \cite{MR1641058}, Nitaj considered families of
elliptic curves $E\!\left(  a,b\right)  $ with the property that $ab\left(
a+b\right)  $ divided the discriminant of $E\!\left(  a,b\right)  $. For these
families, he developed algorithms for constructing good elliptic curves with
large modified Szpiro ratio from good $ABC$ triples $\left(  a,b,c\right)  $.
In what follows, we also consider good $ABC$ triples $\left(  a,b,c\right)  $
and compute the modified Szpiro ratio of $F_{T}\!\left(  a,b\right)  $ and
$F_{T}\!\left(  b,a\right)  $. We note that $F_{T}\!\left(  a,b\right)  $ and
$F_{T}\!\left(  b,a\right)  $ are $\mathbb{Q}(a,b)$-isomorphic if $T \neq C_1,C_2,C_3,C_7,C_9,C_{10}$. As a consequence, we construct a small database
of good elliptic curves. We note that by Lemma \ref{lemmadisc}, we expect a
considerable number of good elliptic curves to arise from $F_{T}$ in this
fashion. This is further justified by the following results:

\begin{lemma}
\label{Lemmax}For $T\neq C_{5}$, define $z\!\left(  t\right)  =\max\!\left\{
\left\vert A_{T}\!\left(  1,t\right)  ^{3}\right\vert ,B_{T}\!\left(
1,t\right)  ^{2}\right\}  -\hat{D}\left(  1,t\right)  ^{6}$. Set $\xi_{C_{10}%
}=0$ and for $T\neq C_{10}$, let $\xi_{T}=\max\left\{  t\in%
\mathbb{R}
\mid z\!\left(  t\right)  =0\right\}  $, so that%
\[
\xi_{T}\approx\left\{
\begin{array}
[c]{ll}%
0 & \text{if }T\not =C_{3},C_{7},C_{8},C_{9},C_{12}\\
0.1686 & \text{if }T=C_{3},\\
4.3444 & \text{if }T=C_{7},
\end{array}
\right.  \qquad\text{and}\qquad\xi_{T}\approx\left\{
\begin{array}
[c]{ll}%
2.0198 & \text{if }T=C_{8},\\
3.2938 & \text{if }T=C_{9},\\
2.9354 & \text{if }T=C_{12}.
\end{array}
\right.
\]
Then $z\!\left(  t\right)  >0$ on the interval $\left(  \xi_{T},\infty\right)
$. In particular, $\hat{D}\left(  a,b\right)  ^{6}<\max\!\left\{  \left\vert
A_{T}\!\left(  a,b\right)  ^{3}\right\vert ,B_{T}\!\left(  a,b\right)
^{2}\right\}  $ if~$\frac{b}{a}>\xi_{T}$.
\end{lemma}

\begin{proof}
For $T\neq C_5,C_{10}$, \cite[Mathematica\_Verifications.nb]{GitHubGoodEC} verifies that $\xi_T$ is the greatest real root of $z_T(x)$. For $T=C_{10}$, loc. cit. shows that $z_T(x)$ has no real roots. For $T\neq C_5$, it is also shown that $z_T(y)>0$ for some $y>\xi_T$. Consequently, $z_T\!\left(  t\right)  >0$ on the interval $\left(  \xi_{T},\infty\right)$ since each $z_T(x)$ is continuous. This establishes the first statement. The second
statement then follows from the first statement and Lemma \ref{HomPoly}.
\end{proof}

\begin{lemma}
\label{Lemmadatabase}Let $T\neq C_{5}$ and suppose $\left(  a,b,a+b\right)  $
is a good $ABC$ triple with $\frac{b}{a}>\xi_{T}$. Suppose further that $a$ is
even if $T=C_{10}$. If $F_{T}=F_{T}\!\left(  a,b\right)  $ is a semistable
elliptic curve with minimal discriminant $D_{T}=D_{T}\!\left(  a,b\right)  $,
then $F_{T}$ is a good elliptic curve.
\end{lemma}

\begin{proof}
By Proposition \ref{Propmin}, $A_{T}$ and $B_{T}$ are the invariants $c_{4}$
and $c_{6}$ associated to a global minimal model of $F_{T}$. Next, observe
that the assumption that $a$ is even in Lemma \ref{lemmadisc} is only required
if $T=C_{10}$. Therefore $N_{T}=\operatorname{rad}\!\left(  D_{T}\!\left(
a,b\right)  \right)  <\left\vert \hat{D}_{T}\!\left(  a,b\right)  \right\vert
$ where $N_{T}$ is the conductor of $F_{T}$. We conclude that $F_{T}$ is a
good elliptic curve by Lemma \ref{Lemmax}.
\end{proof}

Motivated by Lemma \ref{Lemmadatabase}, we next compute $F_{T}\!\left(
a,b\right)  $ for all good positive $ABC$ triples $\left(  a,b,c\right)  $
with $c\leq 1.2\cdot10^{11}$. This is done by using the data computed by the
$ABC$@Home Project \cite{Smit}, which found all good positive $ABC$ triples
$\left(  a,b,c\right)  $ with $c<10^{18}$. Using this data, we consider
\[
\mathcal{S}=\left\{  \left(  a,b\right)  \mid\left(  a,b,a+b\right)  \text{ is
a good positive }ABC\text{ triple with }a<b<a+b \leq 1.2\cdot10^{11}\right\}  .
\]
From the $ABC$@Home Project, we have that $\left\vert \mathcal{S}\right\vert
=122\ 554$. Now denote by $\left[  E\right]  _{%
\mathbb{Q}
}$, the $%
\mathbb{Q}
$-isomorphism class of an elliptic curve $E/%
\mathbb{Q}
$ and let 
\[\mathcal{D}_{T}=\left\{  \left[  F_{T}\!\left(  X\right)
\right]  _{%
\mathbb{Q}
}\mid X=\left(  a,b\right),\left(  b,a\right) \text{ for } (a,b)\in\mathcal{S}%
\right\}  .
\]
Let $\mathcal{G}_{T}$ be the subset of $\mathcal{D}_{T}$
consisting of good elliptic curves. Below, we list the size of $\mathcal{D}%
_{T}$, $\mathcal{G}_{T}$, and the largest modified Szpiro ratio occurring in
$\mathcal{G}_{T}$, which we denote by $M_{T}=\max\!\left\{  \sigma
_{m}\!\left(  E\right)  \mid\left[  E\right]  _{%
\mathbb{Q}
}\in\mathcal{G}_{T}\right\}  $.%
\[
{\renewcommand*{\arraystretch}{1.2}%
\begin{tabular}
[c]{cccccccc}\hline
$T$ & $C_{1}$ & $C_{2}$ & $C_{3}$ & $C_{4}$ & $C_{6}$ & $C_{7}$ & $C_{8}%
$\\\hline
$\left\vert \mathcal{D}_{T}\right\vert $ & $249 \ 109$ & $249 \ 086$ & $249 \ 108$ & $124 \ 554$ & $124 \ 554$ & $249 \ 108$ & $124 \ 554$\\\hline
$\left\vert \mathcal{G}_{T}\right\vert $ & $237 \ 497$ & $249 \ 086$ & $233 \ 667$ & $108 \ 070$ & $115 \ 103$ & $228 \ 468$ & $124 \ 088$\\\hline
$M_{T}$ & $7.56585$ & $7.52493$ & $6.92343$ & $6.87784$ & $7.06959$ & $7.22651$ & $7.10600$\\\hline
&  &  &  &  &  &  & \\\hline
$T$ & $C_{9}$ & $C_{10}$ & $C_{12}$ & $C_{2}\times C_{2}$ & $C_{2}\times C_{4}$ & $C_{2}\times C_{6}$ & $C_{2}\times C_{8}$\\\hline
$\left\vert \mathcal{D}_{T}\right\vert $ & $249 \ 108$ & $249 \ 108$ & $124 \ 554$ & $124 \ 543$ & $124 \ 543$ & $124 \ 554$ & $124 \ 543$\\\hline
$\left\vert \mathcal{G}_{T}\right\vert $ & $232 \ 843$ & $206 \ 026$ & $124 \ 242$ & $124 \ 543$ & $124 \ 543$ & $114 \ 631$ & $124 \ 543$\\\hline
$M_{T}$ & $6.92320$ & $7.31522$ & $6.96451$ & $7.36206$ & $7.20046$ & $7.06734$ & $7.10984$\\\hline
\end{tabular}
\ }%
\]

\noindent Consequently, there are $2 \ 491 \ 026$ $\mathbb{Q}$-isomorphism classes of elliptic curves in $\mathcal{D}=\bigcup_{T}\mathcal{D}_{T}$. Of these, $2 \ 347 \ 350$ are good elliptic curves. Thus $\approx 94.23 \%$ of the elliptic curves in $\mathcal{D}$ are good. Figure \ref{fi:Histo} summarizes the distribution of the modified Szpiro ratio of all elliptic curves in $\mathcal{D}$. We note that the bin size of the histogram is set to $20 \ 000$. Lastly, Figure~\ref{fi:Scatter} plots the quality of $q(a,b,a+b)$ for $(a,b) \in \mathcal{S}$ vs. the modified Szpiro ratios of $F_T(a,b)$ and $F_T(b,a)$. Of note is that each elliptic curve in $\mathcal{D}$ that arises from an $(a,b) \in \mathcal{S}$ with $q(a,b,a+b) >1.2281$ is good. That being said, the number of good $ABC$ triples in $\mathcal{S}$ satisfying this inequality is $1\ 229$.

\noindent \textbf{Acknowledgments.} The author would like to thank Edray
Goins for his helpful
comments on this work. The author also thanks the referee for their helpful comments and
suggestions.

\begin{figure}[H]
\caption{Histogram of $\sigma_{m}(E)$ for $\left[  E\right]  _{\mathbb{Q} }
\in \mathcal{D}$}%
\label{fi:Histo}
\centering
\includegraphics[scale=0.9]{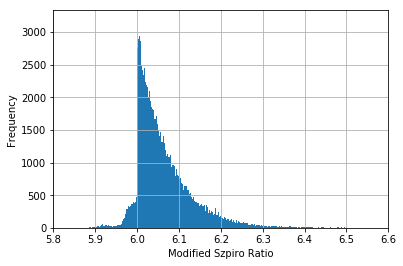}\end{figure}

\begin{figure}[H]
\caption{Scatterplot of $q(a,b,a+b)$ for $(a,b)\in \mathcal{S}$ vs. $\sigma_m (F_T(a,b))$ and $\sigma_m (F_T(b,a))$}
\label{fi:Scatter}
\centering
\includegraphics[scale=0.95]{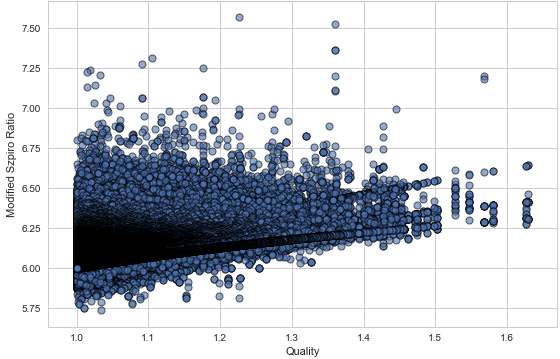}\end{figure}

\newpage

\section{\texorpdfstring{$F_{T}$}{} and its Associated Invariants}\label{SectionTables}
\vspace{-2em}
{\renewcommand*{\arraystretch}{1.2}
\begin{longtable}{C{0.6in}C{1in}C{1.6in}C{2in}}
\caption[Weierstrass Model of $F_{T}$]{The Weierstrass model for $F_{T}$ where\\
$F_{T}:y^{2}=x^{3}+w_{T,1}x^{2}+w_{T,2}x$ if $T=C_{2},C_{2}\times C_{2}$\\
$F_{T}:y^{2}+w_{T,1}xy+w_{T,2}y=x^{3}+v_{T}w_{T,2}x^{2}$ if $T\neq C_{1}%
,C_{2},C_{2}\times C_{2}$\\
$F_{T}:y^{2}+w_{C_{9},1}xy+w_{C_{9},2}y=x^{3}+v_{T}w_{C_{9},2}x^{2}%
+w_{T,1}x+w_{T,2}$ if $T=C_{1}$}\\
\hline
$T$ & $v_{T}$ & $w_{T,1}$ & $w_{T,2}$ \\
\hline
\endfirsthead
\caption[]{\emph{continued}}\\
\hline
$T$ & $v_{T}$ & $w_{T,1}$ & $w_{T,2}$\\
\hline
\endhead
\hline
\multicolumn{2}{r}{\emph{continued on next page}}
\endfoot
\hline
\endlastfoot
$C_{1}$ & $a^{3}$ & $5ab(a+b)(a^{9}-8a^{7}b^{2}-25a^{6}b^{3}-47a^{5}%
b^{4}-61a^{4}b^{5}-53a^{3}b^{6}-28a^{2}b^{7}-9ab^{8}-b^{9})$ & $ab(a+b)(a^{15}%
-13a^{14}b-65a^{13}b^{2}-145a^{12}b^{3}-266a^{11}b^{4}-608a^{10}%
b^{5}-1453a^{9}b^{6}-2719a^{8}b^{7}-3721a^{7}b^{8}-3778a^{6}b^{9}%
-2878a^{5}b^{10}-1621a^{4}b^{11}-643a^{3}b^{12}-162a^{2}b^{13}-23ab^{14}%
-b^{15})$\\\hline
$C_{2}$ & $1$ & $-2(a^{8}-24a^{7}b+20a^{6}b^{2}-24a^{5}b^{3}-26a^{4}%
b^{4}+24a^{3}b^{5}+20a^{2}b^{6}+24ab^{7}+b^{8})$ & $(a^{2}+2ab-b^{2})^{8}$\\\hline
$C_{3}$ & $1$ & $(a^{3}-3a^{2}b-6ab^{2}-b^{3})^{3}$ & $-ab(a+b)(a^{2}%
+ab+b^{2})^{3}(a^{3}-3a^{2}b-6ab^{2}-b^{3})^{6}$\\\hline
$C_{4}$ & $(a-b)^{2}(a+b)^{6}$ & $(a-b)^{2}(a+b)^{6}$ & $ab(a-b)^{2}%
(a+b)^{6}(a^{2}-ab+b^{2})^{3}$\\\hline
$C_{5}$ & $1$ & $-(8b^{4}-1)(4096b^{16}+512b^{12}+64b^{8}+8b^{4}+1)$ &
$-32768b^{20}$\\\hline
$C_{6}$ & $(a^{2}-4ab+b^{2})^{2}$ & $2(a^{4}-4a^{3}b+10a^{2}b^{2}%
-4ab^{3}+b^{4})$ & $-8ab(a-b)^{2}(a^{2}+b^{2})^{2}$\\\hline
$C_{7}$ & $a^{2}$ & $a^{2}-ab-b^{2}$ & $-ab(a+b)^{2}$\\\hline
$C_{8}$ & $b^{2}(a^{2}+b^{2})$ & $-(a^{4}-2a^{2}b^{2}-b^{4})$ & $-a^{2}%
b^{4}(a-b)(a+b)$\\\hline
$C_{9}$ & $a^{3}$ & $a^{3}-a^{2}b-2ab^{2}-b^{3}$ & $-ab(a+b)^{2}%
(a^{2}+ab+b^{2})$\\\hline
$C_{10}$ & $a(b^{2}-ab-a^{2})$ & $-(a^{3}-4ab^{2}-2b^{3})$ &
$-ab(a+2b)(a+b)^{3}$\\\hline
$C_{12}$ & $-b^{3}\left(  a+b\right)  $ & $a^{4}-2a^{3}b+2a^{2}b^{2}%
-2ab^{3}-b^{4}$ & $-ab^{2}(a-b)(a^{2}+b^{2})(a^{2}-ab+b^{2})$\\\hline
$C_{2}\times C_{2}$ & $1$ & $a^{8}-24a^{7}b+20a^{6}b^{2}-24a^{5}b^{3}%
-26a^{4}b^{4}+24a^{3}b^{5}+20a^{2}b^{6}+24ab^{7}+b^{8}$ &
$-16ab(a-b)(a+b)(a^{2}+b^{2})^{2}(a^{2}-2ab-b^{2})^{4}$\\\hline
$C_{2}\times C_{4}$ & $(a^{2}-2ab-b^{2})^{2}$ & $(a^{2}-2ab-b^{2})^{2}$ &
$-ab(a-b)(a+b)(a^{2}+b^{2})^{2}$\\\hline
$C_{2}\times C_{6}$ & $-8(a+b)^{2}(a^{2}-4ab+b^{2})$ & $-16(a^{4}%
-2a^{3}b-2a^{2}b^{2}-2ab^{3}+b^{4})$ & $-512a^{2}b^{2}(a-b)^{2}(a^{2}+b^{2}%
)$\\\hline
$C_{2}\times C_{8}$ & $a(a+b)(a^{2}+2ab-b^{2})$ & $a^{4}+4a^{3}b-b^{4}$ &
$-a^{2}b(-a+b)(a+b)^{2}(a^{2}+b^{2})$%
\label{ta:FTmodel}	
\end{longtable}}

{\renewcommand*{\arraystretch}{1.2} \begin{longtable}{C{0.6in}C{4.9in}}
\caption{The Polynomials $A_{T}$}\\
\hline
$T$ &$A_{T}$\\
\hline
\endfirsthead
\caption[]{\emph{continued}}\\
\hline
$T$ & $A_{T}$\\
\hline
\endhead
\hline
\multicolumn{2}{r}{\emph{continued on next page}}
\endfoot
\hline
\endlastfoot
	
$C_{1}$ & $(a^{3}-3a^{2}b-6ab^{2}-b^{3})(a^{9}-225a^{8}b-855a^{7}%
b^{2}-1866a^{6}b^{3}-2844a^{5}b^{4}-3123a^{4}b^{5}-2265a^{3}b^{6}%
-981a^{2}b^{7}-234ab^{8}-b^{9})$\\\hline
$C_{2}$ & $a^{16}-240a^{15}b+2152a^{14}b^{2}-5040a^{13}b^{3}+4572a^{12}%
b^{4}+1680a^{11}b^{5}-3112a^{10}b^{6}+6480a^{9}b^{7}-6970a^{8}b^{8}%
-6480a^{7}b^{9}-3112a^{6}b^{10}-1680a^{5}b^{11}+4572a^{4}b^{12}+5040a^{3}%
b^{13}+2152a^{2}b^{14}+240ab^{15}+b^{16}$\\\hline
$C_{3}$ & $(a^{3}-3a^{2}b-6ab^{2}-b^{3})(a^{3}+3a^{2}b-b^{3})(a^{6}%
+12a^{5}b+69a^{4}b^{2}+88a^{3}b^{3}+24a^{2}b^{4}-6ab^{5}+b^{6})$\\\hline
$C_{4}$ & $(a^{4}-2a^{3}b-2ab^{3}+b^{4})(a^{12}-6a^{11}b+12a^{10}b^{2}%
-14a^{9}b^{3}+243a^{8}b^{4}-468a^{7}b^{5}+456a^{6}b^{6}-468a^{5}b^{7}%
+243a^{4}b^{8}-14a^{3}b^{9}+12a^{2}b^{10}-6ab^{11}+b^{12})$\\\hline
$C_{5}$ & $1152921504606846976b^{80}-422212465065984b^{60}+15032385536b^{40}%
+393216b^{20}+1$\\\hline
$C_{6}$ & $(a^{4}+4a^{3}b-6a^{2}b^{2}+4ab^{3}+b^{4})(a^{12}-12a^{11}%
b+78a^{10}b^{2}-188a^{9}b^{3}+111a^{8}b^{4}+264a^{7}b^{5}-444a^{6}%
b^{6}+264a^{5}b^{7}+111a^{4}b^{8}-188a^{3}b^{9}+78a^{2}b^{10}-12ab^{11}%
+b^{12})$\\\hline
$C_{7}$ & $(a^{2}+ab+b^{2})(a^{6}+11a^{5}b+30a^{4}b^{2}+15a^{3}b^{3}%
-10a^{2}b^{4}-5ab^{5}+b^{6})$\\\hline
$C_{8}$ & $a^{16}-8a^{14}b^{2}+12a^{12}b^{4}+8a^{10}b^{6}-10a^{8}b^{8}%
+8a^{6}b^{10}+12a^{4}b^{12}-8a^{2}b^{14}+b^{16}$\\\hline
$C_{9}$ & $(a^{3}+3a^{2}b-b^{3})(a^{9}+9a^{8}b+27a^{7}b^{2}+48a^{6}%
b^{3}+54a^{5}b^{4}+45a^{4}b^{5}+27a^{3}b^{6}+9a^{2}b^{7}-b^{9})$\\
$C_{10}$ & $\frac{1}{16}(a^{12}+16a^{11}b+104a^{10}b^{2}+360a^{9}%
b^{3}+720a^{8}b^{4}+816a^{7}b^{5}+416a^{6}b^{6}-96a^{5}b^{7}-240a^{4}%
b^{8}-80a^{3}b^{9}+64a^{2}b^{10}+64ab^{11}+16b^{12})$\\\hline
$C_{12}$ & $(a^{4}-2a^{3}b-2ab^{3}+b^{4})(a^{12}-6a^{11}b+12a^{10}%
b^{2}-14a^{9}b^{3}+3a^{8}b^{4}+12a^{7}b^{5}-24a^{6}b^{6}+12a^{5}b^{7}%
+3a^{4}b^{8}-14a^{3}b^{9}+12a^{2}b^{10}-6ab^{11}+b^{12})$\\\hline
$C_{2}\times C_{2}$ & $a^{16}+232a^{14}b^{2}+732a^{12}b^{4}-1192a^{10}%
b^{6}+710a^{8}b^{8}-1192a^{6}b^{10}+732a^{4}b^{12}+232a^{2}b^{14}+b^{16}$\\\hline
$C_{2}\times C_{4}$ & $(a^{8}-4a^{7}b+4a^{6}b^{2}+28a^{5}b^{3}+6a^{4}%
b^{4}-28a^{3}b^{5}+4a^{2}b^{6}+4ab^{7}+b^{8})(a^{8}+4a^{7}b+4a^{6}%
b^{2}-28a^{5}b^{3}+6a^{4}b^{4}+28a^{3}b^{5}+4a^{2}b^{6}-4ab^{7}+b^{8})$\\\hline
$C_{2}\times C_{6}$ & $(a^{4}-2a^{3}b+6a^{2}b^{2}-2ab^{3}+b^{4})(a^{12}%
-6a^{11}b+6a^{10}b^{2}+10a^{9}b^{3}+15a^{8}b^{4}-36a^{7}b^{5}+84a^{6}%
b^{6}-36a^{5}b^{7}+15a^{4}b^{8}+10a^{3}b^{9}+6a^{2}b^{10}-6ab^{11}+b^{12})$\\\hline
$C_{2}\times C_{8}$ & $a^{16}-8a^{14}b^{2}+12a^{12}b^{4}+8a^{10}b^{6}%
+230a^{8}b^{8}+8a^{6}b^{10}+12a^{4}b^{12}-8a^{2}b^{14}+b^{16}$%
\label{ta:AT}	
\end{longtable}}

{\renewcommand*{\arraystretch}{1.2} \begin{longtable}{C{0.6in}C{4.9in}}
\caption{The Polynomials $B_{T}$}\\
\hline
$T$ & $B_{T}$\\
\hline
\endfirsthead
\caption[]{\emph{continued}}\\
\hline
$T$ & $B_{T}$ \\
\hline
\endhead
\hline
\multicolumn{2}{r}{\emph{continued on next page}}
\endfoot
\hline
\endlastfoot
	
$C_{1}$ & $-(a^{18}+522a^{17}b-8433a^{16}b^{2}-56382a^{15}b^{3}-174843a^{14}%
b^{4}-433494a^{13}b^{5}-1084008a^{12}b^{6}-2541474a^{11}b^{7}-4836168a^{10}%
b^{8}-7036328a^{9}b^{9}-7787457a^{8}b^{10}-6599304a^{7}b^{11}-4265121a^{6}%
b^{12}-2050470a^{5}b^{13}-692973a^{4}b^{14}-148722a^{3}b^{15}-17154a^{2}%
b^{16}-504ab^{17}+b^{18})$\\\hline
$C_{2}$ & $-(a^{8}-24a^{7}b+20a^{6}b^{2}-24a^{5}b^{3}-26a^{4}b^{4}%
+24a^{3}b^{5}+20a^{2}b^{6}+24ab^{7}+b^{8})(a^{16}+528a^{15}b-3992a^{14}%
b^{2}+11088a^{13}b^{3}-7716a^{12}b^{4}-3696a^{11}b^{5}+3032a^{10}%
b^{6}-14256a^{9}b^{7}+17606a^{8}b^{8}+14256a^{7}b^{9}+3032a^{6}b^{10}%
+3696a^{5}b^{11}-7716a^{4}b^{12}-11088a^{3}b^{13}-3992a^{2}b^{14}%
-528ab^{15}+b^{16})$\\\hline
$C_{3}$ & $-(a^{6}+12a^{5}b+15a^{4}b^{2}-20a^{3}b^{3}-30a^{2}b^{4}%
-6ab^{5}+b^{6})(a^{12}+6a^{11}b+48a^{10}b^{2}+428a^{9}b^{3}+1899a^{8}%
b^{4}+3636a^{7}b^{5}+3030a^{6}b^{6}+720a^{5}b^{7}-288a^{4}b^{8}-58a^{3}%
b^{9}+48a^{2}b^{10}+6ab^{11}+b^{12})$\\\hline
$C_{4}$ & $-(a^{8}-4a^{7}b+28a^{6}b^{2}-52a^{5}b^{3}+46a^{4}b^{4}-52a^{3}%
b^{5}+28a^{2}b^{6}-4ab^{7}+b^{8})(a^{16}-8a^{15}b+104a^{13}b^{3}%
-220a^{12}b^{4}+216a^{11}b^{5}-728a^{10}b^{6}+1144a^{9}b^{7}-1026a^{8}%
b^{8}+1144a^{7}b^{9}-728a^{6}b^{10}+216a^{5}b^{11}-220a^{4}b^{12}%
+104a^{3}b^{13}-8ab^{15}+b^{16})$\\\hline
$C_{5}$ & $-(8b^{4}-4b^{2}+1)(8b^{4}+4b^{2}+1)(4096b^{16}-2048b^{14}%
+512b^{12}-64b^{8}+8b^{4}-4b^{2}+1)(4096b^{16}+2048b^{14}+512b^{12}%
-64b^{8}+8b^{4}+4b^{2}+1)(1152921504606846976b^{80}-633318697598976b^{60}%
+79456894976b^{40}+589824b^{20}+1)$\\\hline
$C_{6}$ & $-(a^{8}-4a^{7}b+4a^{6}b^{2}+20a^{5}b^{3}-26a^{4}b^{4}+20a^{3}%
b^{5}+4a^{2}b^{6}-4ab^{7}+b^{8})(a^{16}-8a^{15}b+24a^{14}b^{2}-568a^{13}%
b^{3}+2684a^{12}b^{4}-4776a^{11}b^{5}+2344a^{10}b^{6}+4840a^{9}b^{7}%
-8826a^{8}b^{8}+4840a^{7}b^{9}+2344a^{6}b^{10}-4776a^{5}b^{11}+2684a^{4}%
b^{12}-568a^{3}b^{13}+24a^{2}b^{14}-8ab^{15}+b^{16})$\\\hline
$C_{7}$ & $-(a^{12}+18a^{11}b+117a^{10}b^{2}+354a^{9}b^{3}+570a^{8}%
b^{4}+486a^{7}b^{5}+273a^{6}b^{6}+222a^{5}b^{7}+174a^{4}b^{8}+46a^{3}%
b^{9}-15a^{2}b^{10}-6ab^{11}+b^{12})$\\\hline
$C_{8}$ & $-(a^{8}-4a^{6}b^{2}-2a^{4}b^{4}-4a^{2}b^{6}+b^{8})(a^{16}%
-8a^{14}b^{2}+12a^{12}b^{4}+8a^{10}b^{6}-34a^{8}b^{8}+8a^{6}b^{10}%
+12a^{4}b^{12}-8a^{2}b^{14}+b^{16})$\\\hline
$C_{9}$ & $-(a^{18}+18a^{17}b+135a^{16}b^{2}+570a^{15}b^{3}+1557a^{14}%
b^{4}+2970a^{13}b^{5}+4128a^{12}b^{6}+4230a^{11}b^{7}+3240a^{10}%
b^{8}+2032a^{9}b^{9}+1359a^{8}b^{10}+1080a^{7}b^{11}+735a^{6}b^{12}%
+306a^{5}b^{13}+27a^{4}b^{14}-42a^{3}b^{15}-18a^{2}b^{16}+b^{18})$\\\hline
$C_{10}$ & $-\frac{1}{64}(a^{2}+2ab+2b^{2})(a^{4}+6a^{3}b+6a^{2}%
b^{2}-4ab^{3}-4b^{4})(a^{4}+6a^{3}b+12a^{2}b^{2}+8ab^{3}+2b^{4}%
)(a^{8}+10a^{7}b+32a^{6}b^{2}+40a^{5}b^{3}+14a^{4}b^{4}+8a^{2}b^{6}-4b^{8})$\\\hline
$C_{12}$ & $-(a^{8}-4a^{7}b+4a^{6}b^{2}-4a^{5}b^{3}-2a^{4}b^{4}-4a^{3}%
b^{5}+4a^{2}b^{6}-4ab^{7}+b^{8})(a^{16}-8a^{15}b+24a^{14}b^{2}-40a^{13}%
b^{3}+44a^{12}b^{4}-24a^{11}b^{5}-32a^{10}b^{6}+88a^{9}b^{7}-114a^{8}%
b^{8}+88a^{7}b^{9}-32a^{6}b^{10}-24a^{5}b^{11}+44a^{4}b^{12}-40a^{3}%
b^{13}+24a^{2}b^{14}-8ab^{15}+b^{16})$\\\hline
$C_{2}\times C_{2}$ & $-(a^{8}+20a^{6}b^{2}-26a^{4}b^{4}+20a^{2}b^{6}%
+b^{8})(a^{8}-24a^{7}b+20a^{6}b^{2}-24a^{5}b^{3}-26a^{4}b^{4}%
+24a^{3}b^{5}+20a^{2}b^{6}+24ab^{7}+b^{8})(a^{8}+24a^{7}b+20a^{6}b^{2}%
+24a^{5}b^{3}-26a^{4}b^{4}-24a^{3}b^{5}+20a^{2}b^{6}-24ab^{7}+b^{8})$\\\hline
$C_{2}\times C_{4}$ & $-(a^{4}-4a^{3}b-6a^{2}b^{2}+4ab^{3}+b^{4})
(a^{4}+4a^{3}b-6a^{2}b^{2}-4ab^{3}+b^{4})(a^{8}-4a^{6}b^{2}+22a^{4}%
b^{4}-4a^{2}b^{6}+b^{8})(a^{8}+20a^{6}b^{2}-26a^{4}b^{4}+20a^{2}b^{6}+b^{8}%
)$\\\hline
$C_{2}\times C_{6}$ & $-(a^{8}-4a^{7}b+4a^{6}b^{2}-28a^{5}b^{3}+22a^{4}%
b^{4}-28a^{3}b^{5}+4a^{2}b^{6}-4ab^{7}+b^{8})(a^{8}-4a^{7}b+4a^{6}b^{2}%
-4a^{5}b^{3}-2a^{4}b^{4}-4a^{3}b^{5}+4a^{2}b^{6}-4ab^{7}+b^{8})(a^{8}%
-4a^{7}b+4a^{6}b^{2}+20a^{5}b^{3}-26a^{4}b^{4}+20a^{3}b^{5}+4a^{2}%
b^{6}-4ab^{7}+b^{8})$\\\hline
$C_{2}\times C_{8}$ & $-(a^{8}-4a^{6}b^{2}-26a^{4}b^{4}-4a^{2}b^{6}%
+b^{8})(a^{8}-4a^{6}b^{2}-2a^{4}b^{4}-4a^{2}b^{6}+b^{8})(a^{8}-4a^{6}%
b^{2}+22a^{4}b^{4}-4a^{2}b^{6}+b^{8})$
\label{ta:BT}	
\end{longtable}}

{\renewcommand*{\arraystretch}{1.2} \begin{longtable}{C{0.6in}C{4.9in}}
\caption{The Polynomials $D_{T}$}\\
\hline
$T$ & $D_{T} $\\
\hline
\endfirsthead
\caption[]{\emph{continued}}\\
\hline
$T$ &  $D_{T}$ \\
\hline
\endhead
\hline
\multicolumn{2}{r}{\emph{continued on next page}}
\endfoot
\hline
\endlastfoot
	
$C_{1}$ & $-ab(a+b)(a^{2}+ab+b^{2})^{3}(a^{3}+6a^{2}b+3ab^{2}-b^{3})^{9}$\\\hline
$C_{2}$ & $-ab(a-b)(a+b)(a^{2}+b^{2})^{2}(a^{2}-2ab-b^{2})^{4}(a^{2}%
+2ab-b^{2})^{16}$\\\hline
$C_{3}$ & $-a^{3}b^{3}(a+b)^{3}(a^{2}+ab+b^{2})^{9}(a^{3}+6a^{2}%
b+3ab^{2}-b^{3})^{3}$\\\hline
$C_{4}$ & $a^{4}b^{4}(a+b)^{6}(a-b)^{2}(a^{2}+b^{2})(a^{2}-4ab+b^{2}%
)^{3}(a^{2}-ab+b^{2})^{12}$\\\hline
$C_{5}$ & $2^{75}b^{100}(64b^{8}-8b^{4}-1)(4096b^{16}-1024b^{12}%
+256b^{8}-24b^{4}+1)(4096b^{16}+1536b^{12}+256b^{8}+16b^{4}+1)$\\\hline
$C_{6}$ & $a^{3}b^{3}(a+b)^{2}(a-b)^{6}(a^{2}-ab+b^{2})(a^{2}-4ab+b^{2}%
)^{4}(a^{2}+b^{2})^{12}$\\\hline
$C_{7}$ & $-a^{7}b^{7}(a+b)^{7}(a^{3}+8a^{2}b+5ab^{2}-b^{3})$\\\hline
$C_{8}$ & $a^{16}b^{16}(a+b)^{4}(a-b)^{4}(a^{2}-2ab-b^{2})(a^{2}%
+2ab-b^{2})(a^{2}+b^{2})^{2}$\\\hline
$C_{9}$ & $-a^{9}b^{9}(a+b)^{9}(a^{2}+ab+b^{2})^{3}(a^{3}+6a^{2}%
b+3ab^{2}-b^{3})$\\\hline
$C_{10}$ & $\frac{1}{4096}a^{5}b^{10}(a+b)^{10}(a+2b)^{5}(a^{2}+6ab+4b^{2}%
)(a^{2}+ab-b^{2})^{2}$\\\hline
$C_{12}$ & $a^{12}b^{12}(a+b)^{2}(a-b)^{6}(a^{2}-4ab+b^{2})(a^{2}+b^{2}%
)^{3}(a^{2}-ab+b^{2})^{4}$\\\hline
$C_{2}\times C_{2}$ & $a^{2}b^{2}(a+b)^{2}(a-b)^{2}(a^{2}+b^{2})^{4}%
(a^{2}-2ab-b^{2})^{8}(a^{2}+2ab-b^{2})^{8}$\\\hline
$C_{2}\times C_{4}$ & $a^{4}b^{4}(a+b)^{4}(a-b)^{4}(a^{2}-2ab-b^{2})^{4}%
(a^{2}+2ab-b^{2})^{4}(a^{2}+b^{2})^{8}$\\\hline
$C_{2}\times C_{6}$ & $a^{6}b^{6}(a+b)^{4}(a-b)^{12}(a^{2}-4ab+b^{2}%
)^{2}(a^{2}-ab+b^{2})^{2}(a^{2}+b^{2})^{6}$\\\hline
$C_{2}\times C_{8}$ & $a^{8}a^{8}(a+b)^{8}(a-b)^{8}(a^{2}-2ab-b^{2})^{2}%
(a^{2}+2ab-b^{2})^{2}(a^{2}+b^{2})^{4}$
\label{ta:DT}	
\end{longtable}}

{\renewcommand*{\arraystretch}{1.2} \begin{longtable}{C{1.3in}C{4.2in}}
\caption{The Polynomials $\hat{D}_{T}$}\\
\hline
$T$ & $\hat{D}_{T} $\\
\hline
\endfirsthead
\caption[]{\emph{continued}}\\
\hline
$T$ &  $\hat{D}_{T}$ \\
\hline
\endhead
\hline
\multicolumn{2}{r}{\emph{continued on next page}}
\endfoot
\hline
\endlastfoot
	
$C_{1},C_{3},C_{9}$ & $-(a+b)(a^{2}+ab+b^{2})(a^{3}+6a^{2}b+3ab^{2}-b^{3})$\\\hline
$C_{2},C_{8},C_{2}\times C_{2},$   $C_{2}\times C_{4},C_{2}\times C_{8}$ &
$-(a-b)(a+b)(a^{2}+b^{2})(a^{2}-2ab-b^{2})(a^{2}+2ab-b^{2})$\\\hline
$C_{4},C_{6},C_{12},C_{2}\times C_{6}$ & $-(a+b)(a-b)(a^{2}+b^{2}%
)(a^{2}-4ab+b^{2})(a^{2}-ab+b^{2})$\\\hline
$C_{5}$ & $\frac{2}{5}(64b^{8}-8b^{4}-1)(4096b^{16}-1024b^{12}+256b^{8}%
-24b^{4}+1)(4096b^{16}+1536b^{12}+256b^{8}+16b^{4}+1)$\\\hline
$C_{7}$ & $-(a+b)(a^{3}+8a^{2}b+5ab^{2}-b^{3})$\\\hline
$C_{10}$ & $\frac{-1}{8}(a+b)(a+2b)(a^{2}+6ab+4b^{2})(a^{2}+ab-b^{2})$%
\label{ta:DhatT}	
\end{longtable}}

\bibliographystyle{amsplain}
\bibliography{bibliography}

\end{document}